\documentclass[11pt]{amsart}
\usepackage{graphicx}
\usepackage{amssymb}
\usepackage{cite}
\usepackage{hyperref}
\usepackage{enumerate}
\usepackage{txfonts} 
\usepackage{mathtools}
\usepackage[usenames,dvipsnames]{color}

\numberwithin{equation}{section} 

\newtheorem{theorem}{Theorem}[section]
\newtheorem{theorem*}{Theorem}
\newtheorem{lemma}[theorem]{Lemma}
\newtheorem{proposition}[theorem]{Proposition}
\newtheorem{corollary}[theorem]{Corollary}

\newtheorem{remark}[theorem]{Remark}
\newtheorem{remark*}[theorem*]{Remark}
\newtheorem{definition}[theorem]{Definition}

\def\R{{\mathbb R}}

\def\S{{\mathbb S}}

\def\cL{{\mathcal L}}

\def\cB{{\mathcal B}}

\def\pv{\mbox{P.V.}}

\def\1{\left(}
\def\2{\right)}
\def\3{\left\{}
\def\4{\right\}}
\def\8{\infty}

\DeclareMathOperator*{\osc}{osc}

\begin{document}
\title{The Nonlocal Inverse Problem of Donsker and Varadhan}

\author[]{Gonzalo D\'avila}
\address{
	Gonzalo D\'avila:
Departamento de Matem\'atica, Universidad T\'ecnica Federico Santa Mar\'ia \\
Casilla: v-110, Avda. Espa\~na 1680, Valpara\'iso, Chile
}
\email{gonzalo.davila@usm.cl}

\author[]{Erwin Topp}
\address{
	Erwin Topp:
	Departamento de Matem\'atica y C.C., Universidad de Santiago de Chile,
	Casilla 307, Santiago, CHILE.
	\newline {\tt erwin.topp@usach.cl}
}

\date{\today}

\begin{abstract} 
In this paper we prove  a nonlocal version of the celebrated Inverse Problem of Donsker and Varadhan~\cite{DV} for nonlocal elliptic operators of the form
$$
\cL u = L_K u + \cB_K (h,u),
$$
where $L_K$ is a uniformly elliptic nonlocal operator with smooth coefficients, and, for $h:\R^N \to \R$ smooth and bounded, $\cB_K(h,u)$ is the associated bilinear form, which can be regarded as a nonlocal transport term.
\end{abstract}


\maketitle

\section{Introduction}

In the classic work by Donsker and Varadhan~\cite{DV} the authors were, among other things, interested in the so called ``Inverse Problem" between two linear operators $L_1$ and $L_2$, that is, they studied the relation between two linear operators $L_1$ and $L_2$ whenever the principal eigenvalue of $L_1+V$ and $L_2+V$ coincide for all potentials $V$ and all domains $\Omega$. Here 
\[
L_iu=\text{div}(A_i(x)\nabla u)+b_i(x)\cdot\nabla u
\]
where all the coefficients are assumed to be smooth and the matrix $A_i$ is elliptic.

The authors proved that if for each bounded domain $\Omega$ with smooth boundary and each potential $V \in C^\infty(\R^N)$ we have $\lambda_1(L_1 + V, \Omega) = \lambda_1(L_2 + V, \Omega)$, then the diffusion coefficients must coincide, i.e. $A_1(x)=A_2(x)$, and that the operators $L_1$ and $L_2$ roughly differ by a harmonic function of either $L_2$ or the adjoint $L_2^*$. For a detailed statement of the result, see Theorem~\ref{localDV} below.

The goal of this paper is to study the ``Inverse Problem" between two linear nonlocal operators $\cL_1$ and $\cL_2$ and to prove a nonlocal analogous result to the above mentioned local case.

Let $s \in (0, 1)$, $0 < \gamma \leq \Gamma < +\infty$ and consider symmetric measurable kernels $K: \R^N \times \R^N \to \R$ such that $K(x,y) = K(y,x)$ for all $x,y \in \R^N$. Assume furthermore that $K$ satisfies the ellipticity condition
\begin{equation}\label{eliptic}
\frac{\gamma c_{N, s}}{|x - y|^{N + 2s}} \leq K(x,y) \leq \frac{\Gamma c_{N, s}}{|x -y|^{N + 2s}}, \quad x \neq y,
\end{equation}
where $c_{N,s} > 0$ is a normalizing constant to be precised later.

Given $K$ satisfying \eqref{eliptic}, we define the linear operator $L_K$ by 
\begin{align}\label{LK}
L_Ku(x)=\text{P.V.}\int_{\R^N}(u(y)-u(x))K(x,y)dy.
\end{align}

In the case $K(x,y) = |x - y|^{-(N + 2s)}$, $L_K$ corresponds to $-(-\Delta)^s$, the fractional Laplacian of order $2s$ which we just denote $\Delta^s$ and that can be written as
\[
\Delta^s u(x)=c_{N,s}\text{P.V.}\int_{\R^N}\frac{u(y)-u(x)}{|x-y|^{N+2s}}dy.
\]

The normalizing constant $c_{N, s}$ is chosen such that $\Delta^su \to \Delta u$ for smooth functions with integrable tail.

\medskip

Additionally, given $u, v: \R^N \to \R$ and $K$ a kernel satisfying \eqref{eliptic}, we define 
\begin{align}\label{cB}
\cB_K(u,v)(x)=\frac{1}{2}\int_{\R^N}(u(y)-u(x))(v(y)-v(x))K(x, y)dy.
\end{align}

This operator plays the role of a nonlocal transport term and, in the particular case when $K$ is the kernel of the fractional Laplacian, we have 
\[
\cB(u,v)\to \nabla u \nabla v, \quad \mbox{as} \ s \to 1,
\]
for sufficiently smooth functions $u,v$. Therefore the operator $\cB(u,v)$ can be regarded as the nonlocal analogous of $\nabla u\nabla v$. This operator appears naturally when studying Dirichlet forms associated to Levy processes and in the study of fractional harmonic maps to the sphere, see \cite{DR, DR2,Millot-Sire, BDS, MR}.

For simplicity we will often write $\cB_K(u,u)=\cB_K(u)$, and will omit the dependence on the kernel whenever it is clear from the context. 

Observe that $\cB$ is the \textit{carr\'e  du champ} operator associated to the diffusion process $L_K$, and in the literature is sometimes denoted by $\Gamma$ see \cite{BE, SWZ}. The nonlocal operators $L_K$ and $\cB_K$ are intimately linked through the ``product rule" property
\begin{align}\label{product}
L_K(uv)=u L_Kv + v L_Ku + 2\cB_K(u,v),
\end{align}
as well as by the integration by parts formula
\begin{align}\label{porpartes}
\int_{\R^N}L_K(u)(x)v(x)dx=-\int_{\R^N}\cB_K(u,v)(x)dx,
\end{align}
both valid for smooth functions $u,v$ with certain behavior at infinity depending upon the kernel $K$.

Given $K$ satisfying~\eqref{eliptic} and $h: \R^N \to \R$ bounded and sufficiently smooth, we consider the nonlocal linear operator
\begin{align}\label{LB}
\cL u := L_Ku+\cB_K(u,h),
\end{align}
which will play nonlocal counterpart of the local operator $L$ defined above. Whenever its clear from the context, we will omit the dependence on $K$ and $h$ on $\cL$. 

Given $\Omega \subset \R^N$ bounded domain and a potential $V: \R^N \to \R$, we denote $\lambda_1(\cL + V, \Omega)$ the principal eigenvalue associated to the operator $\cL + V$ on $\Omega$. The main result of this article, the nonlocal version of the inverse problem of Donsker and Varadhan~\cite{DV}, is the following theorem.
\begin{theorem}\label{teo1}
	Let $A_i(\cdot, \cdot)\in \S^n$ be symmetric, positive definite matrices satisfying $A_i(x,y)=A_i(y,x)$ for $i=1,2$. Denote the associated kernels $K_i$ given by
\[
K_i(x,y)=\frac{1}{|(x-y)^tA_i(x,y)(x-y)|^{(N+2s)/2}}
\]
	
For $i=1,2$, let $h_i$ be smooth bounded functions with $\osc_{\R^N}h_i<1$ and denote
\[
\mathcal L_{i} = L_{K_i} + \cB_{K_i}(\cdot, h_i)
\]
	
Assume that for each $\Omega \subset \R^N$ bounded domain with smooth boundary and potential $V \in C^\infty(\R^N)$ we have $\lambda_1(\cL_1 + V, \Omega) = \lambda_1(\cL_2 + V, \Omega)$. Then,
\[
A_1(x_0,x_0)=A_2(x_0,x_0).
\]

In addition, if the matrices $A_i$ can be decomposed as
\begin{align}\label{decomp}
A_i(x,y)=\tilde A_i(x)\tilde A_i(y)+\tilde A_i(y)\tilde A_i(x)
\end{align}
then 
\begin{itemize}
\item[i.-] $\tilde A_1=\tilde A_2$
\item[ii.-] There exists a constant $C$ such that $h_1(x)=h_2(x)+C$.
\end{itemize} 
\end{theorem}

Observe that in general $A(x)A(y)\neq A(y)A(x)$ and therefore the associated kernel $K$ would not be symmetric, which is a key assumption in order to be able to integrate by parts. The decomposition in the statement of the Theorem is the natural symmetric version when the variables $x$ and $y$ are decoupled. Also note that the requirement $\osc_{\R^N}h_i<1$ seems natural in order to have the comparison principle, see Remark \ref{rmkcomparison}. 

Theorem \ref{teo1} is the nonlocal analogous of the result presented in~\cite{DV} in the special case $b=\nabla h$ for some smooth function $h$. For the sake of comparison let us recall the result presented in~\cite{DV}.
\begin{theorem}[Theorem 3.2 \cite{DV}]\label{localDV}
	If $L_1$ and $L_2$ are two operators of the form
	\[
	L_iu=\text{div}(A_1\nabla u)+b_i\cdot\nabla u
	\]
	with $A_i,b_i$ smooth. Assume that for each bounded domain $\Omega$ with smooth boundary and each potential $V \in C^\infty(\R^N)$ we have $\lambda_1(L_1 + V, \Omega) = \lambda_1(L_2 + V, \Omega)$, then either there exists a positive harmonic function $u$ for
	$L_2$ such that $L_1\phi = L_2( u\phi)/u$ or there exists a positive invariant density $\rho$ for $L_2$ (i.e. $L_2^*\rho=0$, $\rho>0$) such that $L_1\phi = L_2^*(\rho\phi)/\rho$.
\end{theorem}

In the proof of this theorem, the authors first prove that the diffusion coefficients are the same, i.e. $ A_1=A_2$, and afterwards they conclude that 
\begin{align}\label{condiciondriftlocal}
\text{div} (b_1)+\langle b_1, A_1^{-1}b_1\rangle=\text{div} (b_2)+\langle b_2, A_2^{-1}b_2\rangle.
\end{align}

From this point, and after several computations, they show that $A^{-1}(b_1-b_2)=\nabla \phi$ for some scalar function $\phi$ that solves a related PDE, which allows them to conclude the theorem.

If we assume, in the local case, that the coefficients $b_i=\nabla h_i$ with $A, \nabla h_i$ bounded we can conclude that $h_1(x)=h_2(x)+C$, for some constant $C$. Indeed, from \eqref{condiciondriftlocal} we have
\[
\Delta h_1 +\langle \nabla h_1, A_1^{-1}\nabla h_1\rangle=\Delta h_2+\langle \nabla h_2, A_2^{-1}\nabla h_2\rangle.
\]
At this point we can assume that $\text{det}(A)=1$ and so $A^t=A^{-1}$. Therefore, if we denote $w=h_1-h_2$ we have that $w$ solves 
\[
\Delta w+\kappa\cdot\nabla w=0, 
\]
where $\kappa=A^{-1} \nabla h_1+ A \nabla h_2$, which is a bounded function. At this point we can apply Liouville's Theorem to conclude $w$ is constant, from here the conclusion follows. Therefore in this particular case the harmonic function $u$ and the invariant measure $\rho$ are just the constant function, and the result becomes the local version.

\medskip

Naturally, the first step in order to prove Theorem \ref{teo1} is the study of the principal eigenvalue problem associated to the nonlocal operator \eqref{LB}. More precisely, given $\Omega \subset \R^N$ a bounded domain and a potential $V \in L^\infty(\Omega)$, we are interested in the existence of nontrivial solutions $(\phi, \lambda)$ with $\phi > 0$ in $\Omega$ for the problem
\begin{equation}\label{eqintro}
\cL \phi + V\phi = -\lambda \phi \quad \mbox{in} \ \Omega; \qquad u = 0 \quad \mbox{in} \ \Omega^c.
\end{equation}

The integration by parts formula~\eqref{porpartes} makes possible to pose this problem in the weak sense (see Section~\ref{secweak} for a precise definition of the fractional Sobolev space $\mathbb H_0^s(\Omega)$). 

Even though the diffusion $L_K$ is symmetric and self-adjoint, the full operator $\cL$ is not, due to the presence of the transport term $\cB_K(\cdot, h)$. This suggests the study of the \textsl{the principal eigenvalue} associated to $\cL + V$ utilizing the ideas of Berestycki, Nirenberg and Varadhan~\cite{BNV}. We exploit the characterization of the principal eigenvalue given by 
\begin{equation*}
\lambda_1 = \lambda_1(\cL + V, \Omega) = \sup \{ \lambda : \ \exists \ \phi > 0 \ \mbox{in} \ \Omega \ s.t. \ \cL \phi + V\phi \leq -\lambda \phi \ \mbox{in} \ \Omega \},
\end{equation*}
where the equation in the definition is satisfied in the weak sense. This definition allows to develop an iterative method leading to the existence of an eigenpair $(\phi_1, \lambda_1)$ with $\phi_1 > 0$ in $\Omega$, solving~\eqref{eqintro}. 

The main novelty here is the consideration of the transport term $\cB$. Under the assumption that $h$ is sufficiently smooth (say, Lipschitz continuous and bounded), the nonlocal term $\cB(\cdot, h)$ can be regarded as a nonlocal operator of ``order" $\max \{ 2s - 1, 0\}$. Then, by means of energy estimates coming from the ellipticity of $L_K$, we can solve a general Dirichlet problem which is the initial step to solve the eigenvalue problem. Here, uniqueness plays an important role, and noting that $\cB$ competes in terms of ellipticity with $L_K$, we need to impose an extra assumption on $h$, that is  
$$
\osc_{\R^N}(h) := \sup_{\R^N}(h) - \inf_{\R^N}(h) < 1.
$$ 

Such an assumption seems to be natural to keep the full operator $\cL$ ``degenerate elliptic". See the discussion about this in Remark~\ref{rmkcomparison}.

Once the existence of the principal eigenvalue is obtained, we provide a second min-max characterization for it, which is crucial in the proof of Theorem \ref{teo1}.

We exploit this min-max formula for the principal eigenvalue through a dual formulation that involves the operator
$$
I(\mu)=-\inf\limits_{\substack {u>0\\
		u\in C^{\infty}(\R^N)}}\int_{\R^N} \left(\frac{\cL u}{u}\right)d\mu,
$$
where $\mu$ is a probability measure in $\R^N$. Basically, if the principal eigenvalues of two operators $\cL_1, \cL_2$ coincide in the sense of Theorem~\ref{teo1} in every domain, then the associated operators $I_1, I_2$ must coincide for each $\mu$. Thus, we prove the inverse problem by taking a sequence of measures $\mu = f_\lambda(x) dx$ where $f_\lambda$ is a density approaching Dirac deltas as the parameter $\lambda \to 0$. In doing so, we require to obtain a ``first order expansion" of the operator $I$ and control the error terms, see Lemmas~\ref{lemaerrornolocal} and~\ref{errorpositivo} below. In this task, we mention here that the nonlocal nature of the operator makes the problem significantly different from the second-order case treated in~\cite{DV}, and new ideas must be brought into the analysis.

\medskip

The paper is organized as follows. In Section~\ref{secweak} we present the notion of weak solution, prove the maximum principle and the existence of the principal eigenvalue, as well as its min-max characterization. In Section~\ref{Secsindrift} we study the case without nonlocal transport term, which sheds some light about how to deal with the general operator in Section~\ref{secmainteo}, where we prove our main Theorem~\ref{teo1}. The last section is an Appendix, where we prove some technical results used in the article, and that we present them for completeness.

\section{The eigenvalue problem.}
\label{secweak}

Throughout this section we assume~\eqref{eliptic} holds.

For $\Omega \subseteq \R^N$, $p \geq 1$ and $\sigma \in (0,1)$, we define
$$
W_p^\sigma(\Omega) = \Big{\{}  u \in L^p(\Omega) : [u]_{W_p^\sigma(\Omega)} := \iint_{\Omega \times \Omega} \frac{|u(y) - u(x)|^p}{|x - y|^{N + \sigma p}} dx dy < +\infty \Big{\}},
$$
and we consider the norm $\| u \|_{W_p^\sigma(\Omega)} = \| u \|_{L^p(\Omega)} + [u]_{W_p^\sigma(\Omega)}$. We consider $W_p^\sigma(\Omega)$ the functions $u \in L^p(\Omega)$ making $\| u \|_{W_p^\sigma(\Omega)} < + \infty$. 

Since the Dirichlet problem is at scope, we require the following space $\mathbb W^{\sigma, p}(\Omega)$ the space of measurable functions $u : \R^N \to \R$ such that $u|_\Omega \in L^p(\Omega)$ and that
$$
(x,y) \mapsto \frac{|u(x) - u(y)|^p}{|x - y|^{N + \sigma p}} \in L^1(Q),  
$$
where $Q = \R^{2N} \setminus (\Omega^c \times \Omega^c) = \Omega \times \R^N \cup \Omega^c \times \Omega$. This is a Banach space endowed with the norm
\begin{equation*}
\| u \|_{\mathbb W^{\sigma, p}(\Omega)}^p = \| u \|_{L^p(\Omega)}^p + [u]_{\mathbb W^{\sigma, p}(\Omega)}^p := \| u \|_{L^p(\Omega)}^p + \iint_Q \frac{|u(x) - u(y)|^p}{|x - y|^{N + \sigma p}} dydx.
\end{equation*}

We clearly have $\mathbb W^{\sigma, p}(\Omega) \subset W^{\sigma, p}(\Omega)$.

In the case $p=2$ we write $\mathbb H^{\sigma}(\Omega)$. We write $\mathbb H_0^\sigma(\Omega)$ the completion of $C_c^1(\Omega)$ with the norm in $\mathbb H^\sigma(\Omega)$. It is a Hilbert space with the inner product
\begin{equation*}
(u,v)_{\mathbb H_0^\sigma(\Omega)} = (u,v)_{L^2(\Omega)} + \iint_Q \frac{(u(x) - u(y))(v(x) - v(y))}{|x - y|^{N + 2 \sigma}} dydx
\end{equation*}

For $\sigma\geq 1$, all the above definitions can be extended using weak derivatives, see~\cite{RS}.

In the Dirichlet problem we will study the standard framework is to work in the space $\mathbb H_0^s(\Omega)$. We denote $2_s^* = 2N/(N - 2s)$ the fractional Sobolev critical exponent.

\subsection{Weak solutions.}
In this first part, we are going to consider $\Omega \subset \R^N$ a bounded smooth domain and Dirichlet problems of the form
\begin{align}\label{Dirichlet}
\cL u + Vu = f \quad \mbox{in} \ \Omega; \qquad u = 0 \quad \mbox{in} \ \Omega^c,
\end{align}
where $f \in L^2(\Omega)$, $h \in H^s(\R^N) \cap L^\infty(\R^N)$ and $V \in L^\infty(\Omega)$.

For simplicity of the notation, we omit the index $K$ in $L$ and $\cB$.

We look for weak solutions to~\eqref{Dirichlet}, whose precise notion is provided in the following
\begin{definition}(Weak Solutions)
	We say that $u \in \mathbb H_0^s(\Omega)$ is a weak subsolution to the Dirichlet problem~\eqref{Dirichlet} if
	\begin{equation*}
	\int_{\R^{N}} \cB(u,v) dx - \int_{\R^N} \cB(h,u) vdx + \int_{\R^N} Vuvdx \leq -\int_{\R^N} fvdx,
	\end{equation*}
	for all $v \in C^1_c(\Omega)$ with $v \geq 0$. In analogous way it can be defined weak supersolution and a solution.
\end{definition}

The expression above make sense since $\cB(u,h) \in L^2(\Omega)$. Our first result is a very standard application of Lax-Milgram Theorem in the case $h$ is regular enough.
\begin{proposition}\label{propexistence}
	Let $\Omega$ be a bounded domain, $h \in L^\infty(\R^N) \cap \mathrm{Lip}(\R^N)$ and $V \in L^\infty(\Omega)$. There exists $C_0 \in \R$ such that, for every $C \geq C_0$, and for every $f \in L^2(\Omega)$, there exists a unique weak solution to the problem
	\begin{equation*}
	\cL u + Vu - Cu = f \quad \mbox{in} \ \Omega, \qquad u = 0 \quad \mbox{in} \ \Omega^c.
	\end{equation*}
\end{proposition}

\begin{proof}
	For $u, v \in \mathbb H_0^s(\Omega)$, we define the bilinear form
	\begin{equation*}
	M(u,v) = \int_{\R^{N}} \cB(u,v) dx - \int_{\R^N} \cB(h,u) vdx + \int_{\R^N} (C - V)uvdx.
	\end{equation*}
	
	
	We start assuming $C > \| V \|_\infty$. The proof follows along classical lines as soon as we can control the transport term to conclude the coercivity of $M$ and use Lax-Milgram Theorem. In fact, we see that for $u \in \mathbb H_0^s(\Omega)$ we have
	\begin{align*}
	\int_{\R^N} \cB(h,u) udx = & \int_\Omega \int_{\Omega} (u(y) - u(x))(h(y) - h(x))K(x, y)dy u(x)dx \\
	& - \int_\Omega \int_{\Omega^c}(h(y) - h(x))K(x, y)dy u^2(x)dx \\
	=: & B_1 + B_2. 
	\end{align*}
	
	For $B_2$ we set $\epsilon_0 > 0$ small and $\delta_0 > 0$ such that if $|x - y| \leq \delta_0$, then $|h(x) - h(y)| \leq \epsilon_0$. Thus, we see that
	\begin{align*}
	B_2 \leq & \epsilon_0 \int_\Omega \int_{\Omega^c \cap B_{\delta_0}(x)} K(x, y)dy u^2(x)dx + C C\delta_0^{1 - 2s}  \int_\Omega u^2(x)dx \\
	\leq & \epsilon_0 [u]_{\mathbb H_0^s(\Omega)}^2 + C\delta_0^{1-2s} |u|_{L^2}^2.
	\end{align*}
	
	On the other hand, we see that
	\begin{align*}
	B_1 \leq & \mathrm{Lip}(h) \int_\Omega \int_{\Omega} |u(y) - u(x)| |x - y|K(x, y)dy u(x) dx \\
	\leq & C \mathrm{Lip}(h) \int_{\Omega} \Big{(}\int_{\Omega} \frac{|u(y) - u(x)|^2}{|x, y|^{N + 2s}} dy  \Big{)}^{1/2} \Big{(}\int_{\Omega} |x - y|^{2 -N - 2s} dy  \Big{)}^{1/2} u(x) dx \\
	\leq & C \mathrm{Lip}(h) \int_{\Omega} \Big{(}\int_{\Omega} \frac{|y(y) - u(x)|^2}{|x - y|^{N + 2s}} dy  \Big{)}^{1/2} u(x) dx,
	\end{align*}
	and using Young's inequality, we conclude that
	$$
	B_1 \leq \epsilon_0 [u]_{\mathbb H_0^s(\Omega)}^2 + C \epsilon_0^{-1} |u|_{L^2(\Omega)}^2, 
	$$
	and from this point, using the above estimates, we conclude the coercivity by taking $\epsilon_0$ small and $C_0$ large.
\end{proof}

\medskip

The proper term $C$ in~\eqref{Dirichlet} also plays an important role in the Maximum Principle presented next
\begin{proposition}\label{maxp}
	Assume
	\begin{equation}\label{h0}
	\osc_{\R^N}(h) =: h_0 < 1,
	\end{equation}
	and $C \geq 0$. If $u \in \mathbb H_0^s(\Omega)$ is a weak subsolution to
	\begin{equation*}
	\cL u - Cu\geq 0\quad \mbox{in} \ \Omega;  \qquad u = 0 \quad \mbox{in} \ \Omega^c, 
	\end{equation*}	
	then $u \leq 0$ in $\Omega$.
\end{proposition}

\begin{proof}
	Here we follow the directions of Theorem 8.1 in~\cite{GT}.
	Assume that $u > 0$ at some point in $\Omega$, and let $0 \leq l < \sup_{\Omega} \{ u \}$. Let $v = (u - l)^+ \in \mathbb H_0^s(\Omega)$ and $v \geq 0$. Then, the weak formulation of the problem leads to
	\begin{equation}\label{testineq}
	\int \cB(u,v) dx \leq \int \cB(h,u) vdx,
	\end{equation}
	where we have assumed that $C = 0$ which is the less favorable case.
	
	Now we estimate each term separately. By definition of $v$, we see that
	\begin{align}\label{difusion}
	\int_{\R^{N}} \cB(u,v) dx = 2 I_0 + I_1,
	\end{align}
	where 
	\begin{align*}
	I_0 := & \int_{\Omega} u(x) v(x) \int_{\Omega^c} K(x,y)dy dx \\
	I_1 := & \int_{\Omega^2} (u(y) - u(x))(v(y) - v(x))K(x,y)dydx.
	\end{align*}
	
	Notice that $I_0 \geq 0$. 
	
	For $I_1$, a new splitting of the integral drives us to
	\begin{equation*}
	I_1 = I_{11} + 2I_{12}
	\end{equation*} 
	where 
	\begin{equation}\label{I12}
	\begin{split}
	I_{11} = & \int_{\{ u(x) > l \}} \int_{\{ u(y) > l \}} (u(y) - u(x))(v(y) - v(x)) K(x,y)dy dx \\
	= & \int_{\{ u(x) > l \}} \int_{\{ u(y) > l \}} (v(y) - v(x))^2 K(x,y)dy dx, \\
	I_{12} = & -\int_{\{ u(x) > l \}} v(x) \int_{\{ u(y) \leq l \}} (u(y) - u(x))K(x,y)dydx,
	\end{split}
	\end{equation}
	and where for $\{ u(x) > l\}$ we mean $\{ x \in \Omega : u(x) > l\}$. Notice that both terms are nonnegative.
	
	For a measurable set $A \subset \R^N$ we introduce the notation
	$$
	[v]_{H_K^s(A)}^2 := \int_{A} \int_{A} (v(y) - v(x))^2 K(x,y)dy dx.
	$$
	
	Noticing that $I_{11} = [v]_{H_K^s(u > l)}^2$, from~\eqref{difusion} we conclude that
	\begin{equation}\label{hola}
	\int \cB(u,v)dx \geq 2I_0 + I_{12} + [v]_{H_K^s(u > l)}^2.
	\end{equation}
	
	\medskip

	Now we deal with the nonlocal transport term. We see that
	\begin{align*}
	\int_{\R^N} \cB(h,u) vdx = & \frac{1}{2} \int_{\Omega} \int_{\R^N} (u(y) - u(x))(h(y) - h(x))K(x, y)dy v(x)dx \\
	= & -\frac{1}{2} \int_{\{ u(x) > l \}} u(x) v(x) \int_{\Omega^c} (h(y) - h(x))K(x, y)dy dx \\
	& + \frac{1}{2} \int_{\{ u(x) > l \}} v(x) \int_{\Omega} (u(y) - u(x))(h(y) - h(x))K(x, y)dy dx \\
	=: & B_0 + B_1.
	\end{align*}
	
	Using~\eqref{h0} we deduce
		\begin{align}\label{B0}
		B_0 \leq & \int_{\{ u(x) > l \}} u(x) v(x) \int_{\Omega^c}K(x, y)dy dx = I_0
		\end{align}
		
		Proceeding in a similar way for the term $B_1$, we get
		\begin{align*}
		B_1 = & \int_{\{ u(x) > l \}} v(x) \int_{\{ u(y) > l \}} (u(y) - u(x)) (h(y) - h(x)) K(x, y)dy dx \\
		& + \int_{\{ u(x) > l \}} v(x) \int_{\{ u(y) \leq l \}} (u(y) - u(x)) (h(y) - h(x)) K(x, y)dy dx \\
		=: & B_{11} + B_{12}.
		\end{align*}
		
		Let us estimate $B_{12}$. From \eqref{h0} we have
		\begin{equation}\label{B12}
		B_{12} \leq h_0 \int_{\{ u(x) > l \}} v(x) \int_{\{ u(y) \leq l \}} |u(y) - u(x)| K(x, y)dy dx \leq h_0 I_{12},
		\end{equation}
		where $I_{12}$ is defined in~\eqref{I12}.
	
	For $B_{11}$ we use H\"older inequality to get
	\begin{align*}
	B_{11} \leq & C_h \int_{\{ u(x) > l \}} v(x) \int_{\{ u(y) > l \}} |v(y) - v(x)| |x - y| K(x, y)dy dx \\
	\leq & C_h \int_{\{ u(x) > l \}} v(x) \Big{(}\int_{\{ u(y) > l \}} |v(y) - v(x)|^2K(x, y) dy  \Big{)}^{1/2} \Big{(}\int_{\{ u(y) > l \}} |x - y|^{2} K(x, y) dy  \Big{)}^{1/2} dx \\
	\leq & C \int_{\{ u(x) > l \}} v(x) \Big{(}\int_{\{ u(y) > l \}} |v(y) - v(x)|^2K(x, y) dy  \Big{)}^{1/2} dx,
	\end{align*}
	where $C > 0$ depends on $h$, $\Gamma$ and $\Omega$. Using H\"older again, we conclude
	\begin{equation}\label{B11}
	B_{11} \leq C \| v \|_{L^2(u>l)} [v]_{H^s_K(u>l)}.
	\end{equation}
	
	
	From \eqref{B0},~\eqref{B12} and~\eqref{B11} we conclude that
	\begin{equation*}\label{hola1}
	\int \cB(h,u) vdx \leq h_0 I_{12} + I_0 + C \| v \|_{L^2(u>l)} [v]_{H_K^s(u>l)}
	\end{equation*}
	
	Then, replacing this last estimate and~\eqref{hola} into~\eqref{testineq}, the condition $h_0 < 1$ allows us to absorb the term $I_{12}$ to conclude that
	\begin{equation*}
	[v]_{H^s(u > l)}^2 \leq C \| v \|_{L^2(u>l)} [v]_{H^s(u>l)}.
	\end{equation*}
	
	
	We use the fractional Sobolev inequality to deduce
	\begin{equation*}
	|v|_{L^{\frac{2N}{N - 2s}}(u > l)} \leq C |\{ u > l \}|^{1/N} |v|_{L^{\frac{2N}{N - 2s}}(u > l)},
	\end{equation*}
	from which we conclude that the function $v = v_l$ satisfies $|\{ v > 0 \}| \geq C > 0$, with $C$ independent of $l$. If $\sup_\Omega u = +\infty$, then we arrive at a contradiction with the integrability of $v$. Assuming that $\sup_\Omega u < +\infty$, we conclude that the set of maxima of $u$ has a positive measure and it is contained in the support of each $v_l$, that vanishes. This leads to a contradiction.
\end{proof}


\medskip

\begin{remark}\label{rmkcomparison}
	It is possible to find $h$ smooth and bounded such that the ``purely integro-differential" operator
	$$
	\Delta^s u + \cB (h,u)
	$$
	does not satisfy the maximum principle. In fact, let $u(x) = (1 - x^2)_+^{1 + s}$, $x \in \R$. Then, using the explicit computations in~\cite{D}, we have
	$$
	-\Delta^s u(x) = c(1 - (1 + 2s)x^2)_+, \quad x \in (-1,1).
	$$
	
	Let $h: \R \to \R_+$ be a smooth and bounded function such that $h = 0$ in $(-1,1)$. Then, we see that
	\begin{align*}
	B(h,u)(x) = &  \int_{\R} (u(y) - u(x))(h(y) - h(x)) |x - y|^{-(1 + 2s)}dy \\
	= & - u(x) \int_{[-1,1]^c} h(y) |x - y|^{-(1 + 2s)}dy.
	\end{align*}
	
	Thus, for $x$ such that $1 - (1 + 2s)x^2 \leq 0$ we have
	$$
	-\Delta^s u(x) + B(h,u)(x) \leq 0.
	$$
	
	On the other hand, if $1 > (1 + 2s)x^2$ then there exists $a_s > 0$ such that $u(x) \geq a_s$ for those $x$. Then, for these points $x$ we have
	$$
	-\Delta^s u(x) + B(h,u)(x) \leq c - a_s \int_{[-1,1]^c} h(y) |x - y|^{-(1 + 2s)}dy,
	$$
	and we can take $h$ large enough outside $(-1,1)$ in order to get 
	$$
	-\Delta^s u(x) + B(h,u)(x) \leq 0 \quad \mbox{en} \ (-1,1),
	$$
	with $u = 0$ in $(-1,1)^c$. But $u(0) > 0$ and maximum principle fails.
\end{remark}

The following result is a consequence of H\"older inequality and available interior regularity estimates.
\begin{lemma}\label{lemaLp}
	Let $\Omega$ be a bounded domain, $f \in L^r(\Omega)$ with $r \geq 2$. Let $h \in C^{\alpha}_b(\R^N)$ for some $\alpha > 0$. If $u \in \mathbb H_0^{s}(\Omega)$ is a weak solution to~\eqref{Dirichlet}, then $u \in W_{loc}^{s,p}(\Omega)$ for each $p$ satisfying 
	$
	2 < p < \min \{ r, 2_s^*, p_\alpha \} 
	$
	with 
	$$
	p_\alpha = \frac{2N}{N - 2(2s - \max\{ 0, 2s - \alpha - 1 \})}.
	$$
\end{lemma}

\begin{proof}
	It is easy to see that by the assumptions on $h$ we have $\cB(h,u) \in L^2(\R^N)$ if $u$ is a weak solution to the problem. Then, we have $u$ is in $\mathbb H^s_0(\Omega) \cap H_{loc}^{2s - \epsilon}(\Omega)$ for each $\epsilon \in (0,1)$ by the result of Cozzi~\cite{C}. By Sobolev inequality, we have $u \in L^p(\Omega)$ for all $p \leq 2_s^*$.
	
Recall the following embeddings
	\begin{equation*}
	W^{\sigma}_r 
	\hookrightarrow 
	W_{p}^{\sigma'},
	\end{equation*}
	with $\sigma' < \sigma$ and $r < p$ satisfying
	\begin{equation}\label{embed}
	\sigma - \frac{N}{r} = \sigma' - \frac{N}{p},
	\end{equation} 
	see Runst and Sickel~\cite{RS}, section 2.2.3. 
	We consider $\sigma = 2s - \epsilon$ with $\epsilon$ small enough and $r=2$ in~\eqref{embed}. We have the existence of $\sigma' \in (0,2s)$ such that there exists $p > 2$ satisfying
	\begin{equation*}
	2s - \epsilon - \frac{N}{2} = \sigma' - \frac{N}{p},
	\end{equation*}
	that is
	\begin{equation*}
	p := \frac{2N}{N - 2(2s- (\sigma' + \epsilon))}. 
	\end{equation*}

	
	
	Now we consider $\Omega' \subset \subset \Omega$ and let $4R = \mbox{dist}(\partial \Omega', \partial \Omega) > 0$. For each $x \in \Omega'$ we can write
	\begin{equation*}
	\cB(h,u)(x) = \cB[B_{R}](h,u)(x) + \cB[B_{R}^c](h,u)(x),
	\end{equation*}
	with
	$$
	\cB[A](h,u)(x) \int_{A} \frac{(u(x + z) - u(x))(h(x + z) - h(x))}{|z|^{N + 2s}}dz,
	$$
	for each measurable set $A \subseteq \R^N$.
	
	For the second integral, using Jensen's inequality we have that
	\begin{align*}
	|\cB[B_R^c](h,u)(x)|^p \leq & C \| h \|_\infty^p  \Big{(} \int_{B_R^c} |u(x + z) - u(x)| |z|^{-(N + 2s)}dz \Big{)}^p \\
	\leq & C_R \| h \|_\infty^p \Big{(} \int_{B_R^c(x) \cap \Omega} |u(y)|^p |x - y|^{-(N + 2s)}dy + |u(x)|^p \Big{)},
	\end{align*}
	where we have used that $u = 0$ in $\Omega^c$. Taking $p \leq 2_s^*$, we have $u \in L^p$ and therefore we arrive at
	\begin{equation*}
	\int_{\Omega'} |\cB[B_R^c](h,u)(x)|^p dx \leq C_R \| h \|_\infty^p \| u \|_{L^p(\Omega)}^p.
	\end{equation*}
	
	Now, for the first integral, using H\"older inequality we have
	\begin{equation*}\label{truco}
	\begin{split}
	|\cB[B_R](h,u)(x)|^p \leq &  
	\Big{(} \int_{B_R} |\delta u_x(z)|^p |z|^{-(N + p \sigma')} dz\Big{)} \Big{(} \int_{B_R} (\delta h_x(z))^q |z|^{-(N + q (2s - \sigma'))} dz\Big{)}^{p/q} 
	\\
	\leq & 
	C[h]_{C^\alpha}^p \Big{(} \int_{B_R} |\delta u_x(z)|^p |z|^{-(N + p\sigma')} dz\Big{)} \Big{(} \int_{B_R} |z|^{-N - q (2s - \sigma' - \alpha)} dz\Big{)}^{p/q} 
	\end{split}
	\end{equation*} 
	where $p, q$ are H\"older conjugates. Then, we take $\sigma' \in (\max \{ 0, 2s - \alpha - 1 \}, 2s)$ (sufficiently close to $\max \{ 0, 2s - \alpha - 1\}$ in order to maximize $p$) such that the second integral in the last expression is finite. Then 
	\begin{align*}
	& |\cB(h,u)(x)|^p \\
	\leq & \ C [h]_{C^\alpha}^p \Big{(} \int_{B_R(x) \cap \Omega} |u(y) - u(x)|^p |x - y|^{-(N + p\sigma')} dy + |u(x)|^p \int_{B_R(x) \cap \Omega^c} |x - y|^{-(N + p\sigma')} dy \Big{)},
	\end{align*}
	and using that $u \in W_{loc}^{p,\sigma'}(\Omega)$ and that $\Omega$ is a bounded domain, we arrive at
	\begin{equation*}
	\int_{\Omega'}|\cB(h,u)(x)|^p dx \leq C_R [h]_{C^\alpha}^p [u]_{W^{p, \sigma'}(\Omega' + B_R)}.
	\end{equation*}
	
	Combining the above inequalities we conclude that $\cB(u,h) \in L^p_{loc}(\Omega)$. Then, we have that $u$ is a weak solution to the problem $L_K u = \tilde f$ with $\tilde f \in L_{loc}^p(\Omega)$. 
	By the regularity estimates of~\cite{LPPS}, we conclude that $u \in W^{s, p}_{loc}(\Omega)$.
\end{proof}

\subsection{Principal eigenvalue problem.} Once classical solutions are at hand, we can readily use the standard machinery to conclude the existence of the principal eigenvalue. We denote here
$$
\cL_{V} u := \cL u + Vu.
$$

\begin{proposition}\label{propeigen}
	Assume $K$ satisfying~\eqref{eliptic}, $h \in \mathrm{Lip}(\R^N) \cap L^\infty(\R^N)$ with~\eqref{h0} and $V \in C^\alpha(\Omega) \cap L^\infty(\Omega)$. Then, there exists $(\phi_1, \lambda_1) \in \mathbb H_0^s(\Omega) \cap C^{2s + \alpha}(\Omega) \times (0,+\infty)$ solving the problem 
	\begin{equation}\label{eigen}
	\cL_{V} \phi_1  = - \lambda_1 \phi_1 \quad \mbox{in} \ \Omega; \qquad \phi_1^+ = 0 \quad \mbox{in} \ \Omega^c.
	\end{equation}
	
	Denote $\Phi_+(\Omega) = \{ u \in \mathbb H_0^s(\R^N) \ : \ u > 0 \ \mbox{in} \ \Omega \}$. The number $\lambda_1$ is characterized as
	\begin{equation*}
	\lambda_1 = \sup \{ \lambda : \ \exists \ \phi \in \Phi_+(\Omega) \ s.t. \ \cL_{V} \phi \leq -\lambda \phi \ \mbox{in} \ \Omega \},
	\end{equation*}
	where the inequality in the definition is understood in the weak sense. 
\end{proposition}

First we study the problem
\begin{lemma}\label{lemalambda<}
	Let $f \in L^2(\bar \Omega)$ nonnegative, and consider $\lambda < \lambda_1$. Then, there exists a unique  nonnegative solution weak solution $u \in \mathbb H_0^s(\Omega)$ to the Dirichlet  problem
	\begin{equation}\label{eqlambda}
	\cL_{V} u = -\lambda u - f \quad \mbox{in} \ \Omega; \qquad u = 0 \quad \mbox{in} \ \Omega^c.
	\end{equation}
\end{lemma}

\begin{proof}
	Replacing $\mathcal L_{V} u$ by $\mathcal L_{V} u - Cu$ for $C$ large enough, we can assume that the operator $\mathcal L_{V}$ satisfies the maximum principle and that $\lambda_1 > 0$. By definition, there exists $u_0 \in \Phi_+(\Omega)$ such that $\cL_{V} u_0 \leq -\lambda u_0 - f$ in $\Omega$, in the weak sense. Then, for each $k \geq 0$ we define $u_k \in \mathbb H_0^s(\Omega)$ the unique weak solution to the problem
	\begin{equation*}
	\cL_{V} u_{k + 1} = -\lambda u_k - f \quad \mbox{in} \ \Omega.
	\end{equation*}
	
	By maximum principle we see that $0 \leq u_{k + 1} \leq u_k \leq u_0$ in $\Omega$, and for $k$ large enough the solution, by a bootstrap argument using Lemma~\ref{lemaLp}, we have the sequence $\{ u_k \}_k$ is uniformly bounded in $\mathbb H^s_0(\Omega)$. Then, after extracting a subsequence, we conclude that $u_k \rightharpoonup u$ in $\mathbb H_0^s(\Omega)$ as $k \to \infty$, and the limit function is a weak solution to~\eqref{eqlambda}. 
\end{proof}

Now we are in position present the

\medskip
\noindent
{\bf \textit{Proof of Proposition~\ref{propeigen}:}} 
	Here we follow the arguments of Birindelli and Demengel~\cite{BD}. We assume without loss of generality that $\lambda_1 > 0$. Consider a sequence $\lambda_k \nearrow \lambda_1$, and $u_k$ the unique nonnegative weak solution to
	\begin{equation*}
	\cL_{V} u_k = -\lambda_k u_k - 1 \quad \mbox{in} \ \Omega; \quad u_k = 0 \quad \mbox{in} \ \Omega^c,
	\end{equation*}
	given by Lemma~\ref{lemalambda<}. 
	
	We claim that $\|u_k \|_{\mathbb H_0^s(\Omega)} \to +\infty$ as $k \to \infty$. If not, there exists a constant $C > 0$ such that, up to subsequences, we have $\| u_k \|_{\mathbb H_0^s(\Omega)} \leq C$. Then, we can extract a subsequence converging weakly to some $\bar u \in \mathbb H_0^s(\Omega)$, weak solution to the problem 
	$$
	\cL_{V} \bar u = -\lambda_1 \bar u - 1 \quad \mbox{in} \ \Omega; \qquad \bar u = 0 \quad \mbox{in} \ \Omega^c,
	$$
	but this contradicts the definition of $\lambda_1$ and the claim follows.
	
	Considering $v_k = u_k/\| u_k\|_{\mathbb H_0^s(\Omega)}$ we have $v_k$  weakly converges (up to subsequences) to a nontrivial solution to the eigenvalue problem. 
\qed

\begin{remark}\label{rmkeigen}
	Assume $V \in L^\infty(\Omega) \cap C^\alpha(\Omega)$ and that $\partial\Omega$ is of class $C^2$. Denote 
	$$
	\tilde \Phi_+(\Omega) = \{ u \in \mathbb H_0^s(\R^N) \cap L^\infty(\Omega) \ : \ u > 0 \ \mbox{in} \ \Omega \}
	$$
	and define $\tilde \lambda_1 = \sup \{ \lambda \ : \ \exists \ \phi \in \tilde \Phi_+(\Omega) \ s.t. \ \cL_{V} \phi \leq \lambda \phi \ \mbox{in} \ \Omega\}.$
	
	It is possible to prove the existence of an eigenpair $(\tilde \phi_1, \tilde \lambda_1) \in C^{2s + \alpha}(\Omega) \cap C(\R^N) \times (0, +\infty)$ solving the eigenvalue problem 
	\begin{equation}\label{eigen2}
	\cL_{V} \tilde \phi_1  = - \tilde \lambda_1 \tilde \phi_1 \quad \mbox{in} \ \Omega; \qquad \phi_1^+ = 0 \quad \mbox{in} \ \Omega^c.
	\end{equation}
	
	In fact, it is possible to prove that the principal eigenvalue is simple in the sense that if $(\phi, \tilde \lambda_1)$ solves~\eqref{eigen2}, then $\phi = t \tilde \phi_1$ for some $t \in \R$, and unique in the sense that if $(\phi, \lambda)$ solves~\eqref{eigen2} with $\lambda > \tilde \lambda_1$, then $\phi$ changes sign in $\Omega$. See~\cite{DQT}. Moreover, using the computations in Lemma~\ref{lemarescale}, it is possible to construct barriers leading to the boundary estimate $c d^{s + \epsilon}(x) \leq \tilde \phi_1(x) \leq C d^{s - \epsilon}(x)$ for all $x \in \Omega$.
\end{remark}

%
%

\medskip

We conclude this section with the following
\begin{proposition}\label{Char}
Under the assumptions of Proposition~\ref{propeigen}, the principal eigenvalue can be characterized as
\begin{equation*}
\lambda_1(\cL_{V}, \Omega) = \min_{\mu \in \mathcal P(\Omega)} \sup_{\tiny{\begin{array}{c}\phi \in C^{\sigma + \alpha}(\Omega) \\ \phi > 0 \end{array}}} \int_{\Omega} \frac{-\cL_{V} \phi (x)}{\phi(x)} d\mu(x),
\end{equation*}
where $\mathcal P (\Omega)$ is the set of Radon probability measures in $\Omega$.
\end{proposition}

\begin{proof}
The existence of the eigenvalue allows to get
\begin{equation*}
\lambda_1(\cL_{V}, \Omega) = \sup_{\tiny{\begin{array}{c}\phi \in C^{\sigma + \alpha}(\Omega) \\ \phi > 0 \end{array}}} \inf_{x \in \Omega} \frac{-\cL_{V} \phi (x)}{\phi(x)}
= \sup_{\tiny{\begin{array}{c}\phi \in C^{\sigma + \alpha}(\Omega) \\ \phi > 0 \end{array}}} \inf_{\mu \in \mathcal P(\Omega)} \int_\Omega \frac{-\cL_{V} \phi (x)}{\phi(x)} d\mu(x),
\end{equation*}
where in the last equality we have used that $\cL_{V} \phi$ is continuous. Now, using the change of variables $\phi = e^u$ we see that 
\begin{align*}
\frac{\cL_{V} \phi (x)}{\phi(x)} = & \pv \int_{\Omega} [e^{\delta_xu(y)} - 1]K(x,y)dy + \int_{\Omega} [e^{\delta_x u(y)} - 1]\delta_x h(y) K(x,y)dy \\
& + \int_{\Omega^c} (1 + \delta_x h(y)) K(x,y)dy +  V(x)\\
= &\int_{\Omega} [e^{\delta_xu(y)} - 1 - \mathbf{1}_{B(x,d(x))} Du(x) \cdot (y - x)] (1 + \delta_x h(y))K(x,y)dy  \\
& + b(x) \cdot Du(x) + \int_{\Omega^c} (1 + \delta_x h(y)) K(x,y)dy + V(x),
\end{align*}
where $b(x) = \int_{B_{d(x)}(x)} \delta_x h(y) (y - x)K(x,y)dy$. Notice that $b \in C^\alpha(\Omega, \R^N)$. Thus, using the notation
\begin{align*}
T(u, \mu) = \int_{\Omega} \Big{\{}\int_{\Omega} & [e^{\delta_xu(y)} - 1 - \mathbf{1}_{B(x,d(x))} Du(x) \cdot (y - x)] (1 + \delta_x h(y))K(x,y)dy  \\
& + b(x) \cdot Du(x) + \int_{\Omega^c} (1 + \delta_x h(y)) K(x,y)dy \Big{\}} d\mu(x) + \int_{\Omega} V(x) d\mu(x),
\end{align*}

we see that 
\begin{equation*}
\lambda_1(\cL_{V}, \Omega) = \sup_{u \in C^{\sigma + \alpha}} \inf_{\mu \in \mathcal P(\Omega)} T(u, \mu).
\end{equation*}

The map $T(u, \mu)$ is convex in $u$ and linear in $\mu$. Sion's Theorem~\cite{S} let us conclude the result.
\end{proof}


\section{Variational Formulation of the Eigenvalue Problem for $L_K$}\label{Secsindrift}

Let $\mu$ be a probability measure in $\R^N$ and denote
\begin{align}\label{generalI}
I(\mu)=-\inf\limits_{\substack {u>0\\
u\in C^{\infty}(\R^N)}}\int_{\R^N} \left(\frac{\cL u}{u}\right)d\mu.
\end{align}

Thanks to Proposition~\ref{Char} and following Theorem 2.3 in~\cite{DV} we have that 
for all $\mu \in \mathcal P(\Omega)$ we can write
\begin{equation}\label{charact2}
I(\mu) = - \inf_{V \in C^\infty(\R^N)} \Big{\{} -\lambda_1(\mathcal L_{V}) - \int_{\R^N} V(x) d\mu(x) \Big{ \}}.
\end{equation}

This allows us to move our attention from the eigenvalue $\lambda_1$ to the the map $I(\mu)$ in order to investigate the operator $\cL$.

The aim of this Section is to understand the minimization problem stated in \eqref{generalI} in the special case $h \equiv 0$. This particular case is of special interest, since it will suggest an appropriate change of variables to get some information in the general case (see Lemma \ref{errornolocal}).

%

Next we consider the case when $\mu$ is a measure with density $f$. 
We have the following preliminary result.
\begin{lemma}\label{sinb}
Assume $K$ satisfies~\eqref{eliptic} and denote $L=L_K$. If $\mu$ is a probability measure with such that $\mu(dx)=f(x)dx$, where $f\in C^\infty (\R^N)$, and if $\sqrt f\in C^{2s+\alpha}(\R^N)$ for some $\alpha>0$, then 
\[
I(\mu)=\int_{\R^N}\cB(\sqrt f)dx
\]
\end{lemma}

\begin{proof}
We will compute the first variation of 
\[
\mathcal I(u)=\int_{\R^N}\frac{Lu}{u}f(x)dx,
\]
for $u$ smooth and strictly positive. Since the operator $L$ is linear we have 
\begin{align*}
\left.\frac{d\mathcal{I} (u+tv)}{dt}\right\rvert_{t=0}&=\int_{\R^N}\left(\frac{-v}{u^2}Lu+\frac{1}{u}Lv\right)f(x)dx\\
&=\int_{\R^N}\left(\frac{-v}{u^2}Lu+vL\left(\frac{f}{u}\right)\right)dx.
\end{align*}
Therefore imposing the first order condition
\[
\left.\frac{d\mathcal{I} (u+tv)}{dt}\right\rvert_{t=0}=0\ \  \forall v,
\]
we deduce 
\begin{align}\label{1storder}
\frac{f}{u^2}Lu=L\left(\frac{f}{u}\right).
\end{align}
Now, using \eqref{product} we write
\[
L\left(\frac{f}{u}\right)=\frac{1}{u}Lf+fL\left(\frac{1}{u}\right)+2\cB\left(f,\frac{1}{u}\right),
\]
and therefore we deduce from \eqref{1storder}
\begin{align}\label{eqmin}
\frac{f}{u}Lu=Lf+2u\cB\left(f,\frac{1}{u}\right)+ufL\left(\frac{1}{u}\right).
\end{align}
Observe that the left hand side of the previous equation is the expression defining the infimum in \eqref{generalI}. We decompose now
\begin{align*}
Lf&=L\left(\frac{uf}{u}\right)\\
&=ufL\left(\frac{1}{u}\right)+2\cB\left(uf,\frac{1}{u}\right)+\frac{1}{u}L(uf),
\end{align*}
which implies
\begin{align*}
ufL\left(\frac{1}{u}\right)&=Lf-2\cB\left(uf,\frac{1}{u}\right)-\frac{1}{u}L(uf)\\
&=Lf-2\cB\left(uf,\frac{1}{u}\right)-\frac{1}{u}\left(uLf+fLu+2\cB(u,f)\right)\\
&=Lf-2\cB\left(uf,\frac{1}{u}\right)-Lf-\frac{f}{u}Lu-\frac{2}{u}\cB(u,f)\\
&=-2\cB\left(uf,\frac{1}{u}\right)-\frac{f}{u}Lu-\frac{2}{u}\cB(u,f).
\end{align*}
where in the second line we applied \eqref{product} to $L(uf)$. Use this last relation in \eqref{eqmin} to deduce
\[
\frac{f}{u}Lu=Lf+2u\cB\left(f,\frac{1}{u}\right)-2\cB\left(uf,\frac{1}{u}\right)-\frac{f}{u}Lu-\frac{2}{u}\cB(u,f),
\]
or equivalently 
\begin{align}
\nonumber 2\frac{f}{u}Lu&=Lf+2u\cB\left(f,\frac{1}{u}\right)-2\cB\left(uf,\frac{1}{u}\right)-\frac{2}{u}\cB(u,f)\\
\label{eqmin2}&=Lf+B(u,f),
\end{align}
where 
\[
B(u,f)=2u\cB\left(f,\frac{1}{u}\right)-2\cB\left(uf,\frac{1}{u}\right)-\frac{2}{u}\cB(u,f)
\]
Let us analyze now the nonlocal transport terms. We have 
\begin{align*}
B(u,f)&=u(x)c_{N,s}\int_{\R^N}(f(x)-f(y))\left(\frac{1}{u(x)}-\frac{1}{u(y)}\right)K(x,y)dy\\
&-c_{N,s}\int_{\R^N}(u(x)f(x)-u(y)f(y))\left(\frac{1}{u(x)}-\frac{1}{u(y)}\right)K(x,y)dy\\
&-c_{N,s}\frac{1}{u(x)}\int_{\R^N}(f(x)-f(y)(u(x)-u(y)))K(x,y)dy\\
&=c_{N,s}\int_{\R^N}u(x)(f(x)-f(y))\frac{(u(y)-u(x))}{u(x)u(y)}K(x,y)dy\\
&-c_{N,s}\int_{\R^N}(u(x)f(x)-u(y)f(y))\frac{(u(y)-u(x))}{u(x)u(y)}K(x,y)dy\\
&+c_{N,s}\int_{\R^N}(f(x)-f(y))u(y)\frac{(u(y)-u(x))}{u(x)u(y)}K(x,y)dy.
\end{align*}
Note that all the integrals have common factor $\frac{(u(y)-u(x))}{u(x)u(y)}$. On the other hand 
\begin{align*}
u(x)(f(x)-f(y))-(u(x)f(x)-u(y)f(y))+(f(x)-f(y))u(y)=u(y)f(x)-u(x)f(y),
\end{align*}
and therefore
\begin{align*}
B(u,f)&=c_{N,s}\int_{\R^N}(u(y)f(x)-u(x)f(y))\frac{(u(y)-u(x))}{u(x)u(y)}dyK(x,y)\\
&=-c_{N,s}\int_{\R^N}\left(\frac{f(x)}{u(x)}-\frac{f(y)}{u(y)}\right)(u(x)-u(y))K(x,y)dy\\
&=-2\cB\left(\frac{f}{u},u\right).
\end{align*}
Coming back to \eqref{eqmin2} we get
\begin{align}\label{eqmin3}
2\frac{f}{u}Lu=Lf-2\cB\left(\frac{f}{u},u\right).
\end{align}
Finally note that, by inspection, a solution of \eqref{eqmin3} is given by $u=\sqrt f$, since by \eqref{product} we have (note that $f>0$)
\begin{align*}
Lf&=L(\sqrt f\cdot \sqrt f)\\
&=2\sqrt fL(\sqrt f) +2\cB(\sqrt f).
\end{align*}
Finally using that the inf in \eqref{generalI} is achieved at $\sqrt f$ and integrating by parts we have  
\begin{align*}
I(\mu)&=-\inf \int_{\R^N}\left(\frac{Lu}{u}\right)f(x)dx\\
&=-\int_{\R^N}\left(\frac{L\sqrt f}{\sqrt f}\right)f(x)dx\\
&=\int_{\R^N}\cB(\sqrt f)dx
\end{align*}
\end{proof}
\begin{remark}
It is quite surprising that the minimum is still achieved at the same point as in the local case, that is, $\sqrt f$. Also note that 
\begin{align*}
\lim\limits_{s\to 2}\cB(\sqrt f)&=|\nabla \sqrt f|^2\\
&=\frac{|\nabla f|^2}{4f},
\end{align*}
and therefore the previous lemma recovers the estimates in \cite{DV}.
\end{remark}

\section{Nonlocal Drift case}
\label{secmainteo}


We devote this section to prove our main Theorem~\ref{teo1}. Following the dual variational formulation presented in~\eqref{charact2}, we obtain Theorem~\ref{teo1} as a Corollary of the following theorem.
\begin{theorem}\label{mainteo}
Let $A_i(\cdot, \cdot)\in \S^n$ be symmetric, positive definite matrices satisfying $A_i(x,y)=A_i(y,x)$ for $i=1,2$. Denote the associated kernels $K_i$ given by
\[
K_i(x,y)=\frac{1}{|(x-y)^tA_i(x,y)(x-y)|^{(N+2s)/2}}
\]
For $i=1,2$, let $h_i$ be smooth functions and $\mathcal L_{K_i}$ be defined by \eqref{LB} and assume $\osc_{\R^N}h_i<1$. Finally denote by $I_i(\mu)$ the operator \eqref{generalI} associated to $\mathcal L_{K_i}$. Then, if $I_1(\mu)=I_2(\mu)$ for all  measures $\mu$ having compact support then 
\[
A_1(x_0,x_0)=A_2(x_0,x_0).
\]
If in addition the matrices $A_i$ can be decomposed as
\begin{align}\label{sproduct}
A_i(x,y)=\tilde A_i(x)\tilde A_i(y)+\tilde A_i(y)\tilde A_i(x)
\end{align}
then 
\begin{itemize}
\item[i.-] $\tilde A_1=\tilde A_2$
\item[ii.-] There exists a constant $C$ such that $h_1(x)=h_2(x)+C$.
\end{itemize} 
\end{theorem}

\begin{remark}
Observe that in general $A(x)A(y)\neq A(y)A(x)$ and therefore the decomposition in the statement of Theorem \ref{mainteo} is the symmetrization of the separable variable case. 
\end{remark}

The proof of Theorem \ref{mainteo} consists on several steps. The first step is to establish an estimate of $I(\mu)$ based on the fact that, in the absence of the nonlocal drift ($h=0$), the minimizer is achieved at $\sqrt f$ whenever $\sqrt f$ is smooth, see Lemma \ref{sinb}.  

Throughout this section we will use repeatedly the following identity. Given a symmetric function $k: \R^N \times \R^N \to \R$, then 
\begin{align}\label{symid}
\iint_{\R^{2N}} k(x,y) f(x)dx dy = \frac{1}{2} \iint_{\R^{2N}} k(x,y) (f(x) + f(y))dx dy,
\end{align}
each time $k$ and $f$ satisfy appropriate assumptions to perform Fubini's Theorem.

The next two lemmas are the generalization of Lemma \ref{sinb} and combined are the nonlocal analogous of Lemma 3.3 in \cite{DV}. We point out that Lemma \ref{lemaerrornolocal} and Lemma \ref{errorpositivo} hold for general symmetric  kernels $K$.

\medskip

\begin{lemma}\label{lemaerrornolocal}
Let $\mathcal L$ be given by \eqref{LB} and $I$ as in~\eqref{generalI}, then 
\[
I(\mu)=\int_{\R^N}\left(\cB(\sqrt f)(x)-\frac{1}{2}\cB(f,h))\right)dx-\mathcal{E},
\]
where 
\begin{equation}\label{errornolocal}
\begin{split}
\mathcal{E}=\inf \limits_{\substack {v>0\\
		v\in C^{\infty}(\R^N)}} & \left\{ \frac{1}{2}\int_{\R^{2N}}\frac{(v(x)-v(y))^2}{v(x)v(y)}\sqrt f(x)\sqrt f(y)K(x,y)dydx\right.\\
&\left.+\frac{1}{4}\int_{\R^{2N}}\frac{\sqrt{f}(x)\sqrt{f}(y)(v^2(x)-v^2(y))(h(x)-h(y))}{v(x)v(y)}K(x,y)dydx\right\}
\end{split}
\end{equation}
\end{lemma}

\begin{proof} 
Assume first that $f > 0$ in $\R^N$, the general case will follow by an approximation argument. Let
\begin{align*}
I := \int \frac{L_K u(x)}{u(x)} f(x)dx,
\end{align*}	
and consider the change of variables $u = \sqrt{f} v$ for $v > 0$ smooth. By the product rule~\eqref{product} we have
\begin{equation}\label{I}
\begin{split}
I  
= & \int L_K (\sqrt{f})(x)\sqrt f(x)dx + \int \frac{f}{v}(x) L_K(v)(x)dx + 2 \int \cB(\sqrt{f}, v)(x) \frac{\sqrt f}{v}(x)dx \\
=: & \int L_K (\sqrt{f})(x)\sqrt f(x)dx + I_1.
\end{split}
\end{equation}


Note that the first term in the right hand side is independent of $v$. After integrating by parts we get 
\begin{align}\label{termino2}
\int \sqrt{f}(x)L_K \sqrt{f}(x)dx = -\int\cB_K(\sqrt f)(x)dx. 
\end{align}
We now concentrate on $I_1$ in~\eqref{I}, that is, the integrals depending on $v$. Integrating by parts once again we get
\begin{align*}
\int_{\R^N} f(x)\frac{L_Kv(x)}{v(x)}dx=-\int_{\R^N}\cB_K\left(v,\frac{f}{v}\right)(x)dx.
\end{align*} 

On the other hand, using the symmetry identity \eqref{symid}, we have
\begin{align*}
& 2\int_{\R^N}\cB_K(\sqrt{f},v)(x)\frac{\sqrt{f}}{v}(x)dx \\
= & \int_{\R^{2N}} (v(x)-v(y))(\sqrt{f}(x)-\sqrt{f}(y))\left(\frac{\sqrt{f}}{v}(x)+\frac{\sqrt{f}}{v}(y)\right)K(x,y)dydx.
\end{align*}

Therefore, combining the last two terms we get
\begin{align*}
I_1 = 
&-\frac{1}{2}\int_{\R^{2N}}(v(x)-v(y))\left\{\frac{f(x)}{v(x)}-\frac{f(y)}{v(y)}-(\sqrt{f(x)}-\sqrt{f(y)})\left[\frac{\sqrt{f(x)}}{v(x)}+\frac{\sqrt{f(y)}}{v(y)}\right]\right\} K(x,y)dydx\\ &=-\frac{1}{2}\int_{\R^{2N}}(v(x)-v(y))\left[\frac{\sqrt{f(x)}\sqrt{f(y)}}{v(x)}-\frac{\sqrt{f(x)}\sqrt{f(y)}}{v(y)}\right]K(x,y)dydx\\
\nonumber &=-\frac{1}{2}\int_{\R^{2N}}(v(x)-v(y))\sqrt{f(x)}\sqrt{f(y)}\left(\frac{v(y)-v(x)}{v(y)v(x)}\right)K(x,y)dydx\\
&=\frac{1}{2}\int_{\R^{2N}}\frac{(v(x)-v(y))^2}{v(x)v(y)}\sqrt{f(x)}\sqrt{f(y)}K(x,y)dydx.
\end{align*}

Taking into account this last computation and~\eqref{termino2}, and replacing them into~\eqref{I}, we conclude that
\begin{equation}\label{Idif}
I = -\int_{\R^N} \cB_K(\sqrt f)(x)dx + \frac{1}{2}\int_{\R^{2N}}\frac{(v(x)-v(y))^2}{v(x)v(y)}\sqrt{f}(x)\sqrt{f}(y)dydx.
\end{equation}
	
\medskip

Now we deal with the term involving the nonlocal transport. Recall that $u = \sqrt f v$, we have
\begin{equation*}
II = \int \cB(\sqrt{f}v,h)(x)\frac{\sqrt{f}(x)}{v(x)}dx.
\end{equation*}

Use identity \eqref{symid} to get

\begin{align*}
II =&\frac{1}{4}\int_{\R^{2N}}\left(\frac{\sqrt{f}(x)}{v(x)}+\frac{\sqrt{f}(y)}{v(y)}\right)(\sqrt{f}(x)v(x)-\sqrt{f}(y)v(y))(h(x)-h(y))K(x,y)dydx\\
=&\frac{1}{4}\int_{\R^{2N}}\frac{(\sqrt{f}(x)v(y)+\sqrt{f}(y)v(x))(\sqrt{f}(x)v(x)-\sqrt{f}(y)v(y))(h(x)-h(y))}{v(x)v(y)}K(x,y)dydx\\
=&\frac{1}{4}\int_{\R^{2N}}\frac{\sqrt{f}(x)\sqrt{f}(y)(v^2(x)-v^2(y))(h(x)-h(y))}{v(x)v(y)} K(x, y)dydx\\
&+ \frac{1}{4}\int_{\R^{2N}}(f(x)-f(y))(h(x)-h(y))K(x,y)dydx,
\end{align*}
from which, by the definition of $\cB_K$, we conclude that
\[
II = \frac{1}{4}\int_{\R^{2N}}\frac{\sqrt{f}(x)\sqrt{f}(y)(v^2(x)-v^2(y))(h(x)-h(y))}{v(x)v(y)} K(x,y)dydx + \frac{1}{2}\int_{\R^N}\cB(f,h)dx.
\]

The last equality combined with~\eqref{Idif} concludes the estimate~\eqref{errornolocal}.

When $f$ does not have compact support one can perform the change $u=\sqrt{f+\varepsilon} v$ for $\varepsilon>0$ and proceed in an analogous way.

\end{proof}

A key step in the proof of Theorem \ref{mainteo} is taking a sequence of measures converging to a Dirac delta. In order to pass to the limit we need to able to control the error term in Lemma \ref{lemaerrornolocal}. To do so, we first refine the estimate on the error in Lemma \ref{lemaerrornolocal} and rewrite it as a positive error plus an additional term.

Let us introduce some notation. For any function $u$ and $x, y\in\R^N$ we denote $\delta(u,x,y)=u(x)-u(y)$, and $\delta^2(u,x,y)=(u(x)-u(y))^2$

\begin{lemma}\label{errorpositivo}
Let $\mathcal{E}$ be given by \eqref{errornolocal}. If $\mathrm{osc}_{\R^N} (h) < 1$, then 
\begin{equation}
\begin{split}
\mathcal{E}=&\int_{\R^{2N}}\sqrt{f}(x)\sqrt{f}(y)\delta^2(h,x,y)K(x,y)dydx \\
& -\inf\limits_{w\in C^\infty(\R^N)}\int_{\R^{2N}}\sqrt{f}(x)\sqrt{f}(y)Q(h,w,x,y)K(x,y)dxdy,
\end{split}
\end{equation}
where $Q\geq 0$ is given by
\[
Q(h,w,x,y)=\delta^2(h,x,y)+\sinh(\delta(w,x,y))\delta(h,x,y)+\cosh(\delta(w,x,y))-1.
\]
\end{lemma}

\begin{proof}
Let $v>0$ be any smooth function and rewrite it as $v(x)=e^{w(x)}$ for some smooth function $w$. Then we have 
\begin{align*}
\frac{(v(x)-v(y))^2}{v(x)v(y)}&=\left(\frac{v(x)}{v(y)}+\frac{v(x)}{v(y)}-2\right)\\
&=(e^{w(x)-w(y)} +e^{w(y)-w(x)}-2)\\
&=2(\cosh(\delta(w,x,y))-1).
\end{align*}
On the other hand
\begin{align*}
\frac{(v^2(x)-v^2(y)))}{v(x)v(y)}&=\left(\frac{v(x)}{v(y)}-\frac{v(x)}{v(y)}\right)\\
&=\left(e^{w(x)-w(y)}-e^{w(y)-w(x)}\right)\\
&=2\sinh(\delta(w,x,y)).
\end{align*}

Taking these computations into consideration in formula \eqref{errornolocal} we get
\begin{align*}
\mathcal{E}=-\inf\limits_{w\in C^{\infty}(\R^N)}\int_{\R^{2N}}\Theta(x,y) \sqrt f(x)\sqrt f(y) K(x,y)dxdy,
\end{align*}
where
\[
\Theta(x,y)= \cosh(\delta(w,x,y))-1 + \frac{1}{2}\sinh(\delta(w,x,y))\delta(h,x,y) 
\]

Let $C > 0$, to be fixed, and let us analyze the function 
$$
q(r,h) = \cosh(r) - 1 + \frac{1}{2} \sinh(r)h + Ch^2/2, \quad (r,h) \in \R \times [-1,1].
$$ 

Notice that $q(r,h) \to +\infty$ as $r \to \infty$ uniformly in $h$, and therefore it attains its global minima in $\R \times [-1,1]$. The necessary optimality conditions on $q$ lead to the existence of a tuple $\bar X = (\bar r, \bar h, \mu_1, \mu_2)$ with multipliers $\mu_1, \mu_2 \geq 0$ satisfying the set of equations
\begin{align*}
& \sinh(\bar r) + \frac{1}{2} \cosh(\bar r) \bar h = 0, \\
& \frac{1}{2} \sinh(\bar r) +  C\bar h - \mu_1 + \mu_2 = 0, \\
& \mu_1(\bar h - 1) = 0; \quad \mu_2 (\bar h + 1) = 0.
\end{align*}

Notice that $\bar X = (0,0,0,0)$ is an interior critical point for all $C$, with critical value equal to zero. On the other hand, taking $C > \frac{1}{2\sqrt{3}}$, we have no interior critical point. Therefore we consider such a $C$ from now on, from which we necessarily have $\bar h = \pm 1$. Then, taking $C = 2$, we conclude that the minima of $q$ is equal to zero. We add and subtract $\delta^2(h,x,y)$ inside the square brackets in the last integral expression for $\mathcal E$ to conclude the result.

\end{proof}

The next step is to take a sequence of measures $d\mu$ concentrating to a Dirac delta. This step will allow us to localize the coefficients $A_i(x,y)$ of Theorem \ref{mainteo}.

\begin{lemma}\label{difusionigual}
Assume that the kernels $K_i$ satisfy the general hypothesis of Theorem \ref{mainteo}, that is
\[
K_i(x,y)=\frac{1}{|(x-y)^t A_i(x,y)(x-y)|^{N+2s}},
\]
where $A_i(x,y)=A_i(y,x)$ is a symmetric matrix and also satisfies
\[
c_1|\xi|^2\leq \xi^tA_i(x,y)\xi\leq C_1|\xi|^2.
\] 
Then, under the hypothesis of Theorem \ref{mainteo} we have that 
\[
\int_{\R^{N}}\cB_{A_1(x_0,x_0)}(\sqrt f)(x)dx=\int_{\R^{N}}\cB_{A_2(x_0,x_0)}(\sqrt f)(x)dx,
\]
for every smooth probability density $f$ with compact support with $\sqrt f$ a $C^{2s+\alpha}(\R^N)$ for some $\alpha>0$.

Here 
\[
\int_{\R^{N}}\cB_{A_i(x_0,x_0)}(\sqrt f)(x)dx=\int_{\R^{2N}}\frac{(\sqrt{f}(x)-\sqrt{f}(y))^2dydx}{|(x-y)^t A_i(x_0,x_0)(x-y)|^{N+2s}}.
\]

\end{lemma}
\begin{proof}
Let $f$ be a probability density in $\R^N$ with compact support so that $\sqrt f$ is a $C^{2s+\alpha}$ for some $\alpha>0$. We will prove the result first in the case $x_0=0$. 

For $\lambda>0$ define $f_\lambda(x)=\lambda^{-N}f(x/\lambda)$, which is also a probability density. We need to analyze the behaviour of $I_i(f_\lambda)$ as $\lambda\to0$. For this we use Lemma \ref{errorpositivo} and analyze each term separately. We have
\begin{align*}
\int_{\R^N}\cB_{K_i}(\sqrt{f_\lambda})(x)dx&=\frac{c_{N,s}}{2}\lambda^{-N}\int_{\R^{2N}}\frac{(\sqrt{f}(x/\lambda) -\sqrt{f}(y/\lambda))^2}{|(x-y)^t A_i(x,y)(x-y)|^{N+2s}}dydx\\
&=\frac{c_{N,s}}{2}\lambda^{N}\int_{\R^{2N}}\frac{(\sqrt{f}(x) -\sqrt{f}(y))^2}{|(\lambda x-\lambda y)^t A_i(\lambda x,\lambda y)(\lambda x-\lambda y)|^{N+2s}}dydx\\
&=\frac{c_{N,s}}{2}\lambda^{-2s}\int_{\R^{2N}}\frac{(\sqrt{f}(x) -\sqrt{f}(y))^2}{|(x-y)^t A_i(\lambda x,\lambda y)(x-y)|^{N+2s}}dydx
\end{align*}
On the other hand we have
\begin{align*}
\int_{\R^N}\cB_{K_i}(f_\lambda,h)dx&=\int_{\R^{N}}f_\lambda(x) L_{K_i}h_idx\\
&=\lambda^{-N}\int_{\R^{N}}f(x/\lambda)L_{K_i}h_i(x)dx\\
&=\int_{\R^{N}}f(x)  L_{K_i}h_i(\lambda x)dx,
\end{align*}
The third term can be rewritten as 
\begin{align*}
&\hspace{-40pt}\int_{\R^{2N}}\sqrt{f_\lambda}(x)\sqrt{f_\lambda}(y)\delta^2(h,x,y)K_i(x,y)dydx\\
&\hspace{80pt}=\lambda^{-N}\int_{\R^{2N}}\frac{\sqrt{f}(x/\lambda)\sqrt{f}(y/\lambda)\delta^2(h,x,y)}{|(x-y)^t A_i( x, y)(x-y)|^{N+2s}}dydx\\
&\hspace{80pt}=\lambda^N\int_{\R^{2N}}\frac{\sqrt{f}(x)\sqrt{f}(y)\delta^2(h,\lambda x,\lambda y)}{|(\lambda x-\lambda y)^t A_i(\lambda x,\lambda y)(\lambda x-\lambda y)|^{N+2s}}\\
&\hspace{80pt}=\lambda^{-2s}\int_{\R^{2N}}\frac{\sqrt{f}(x)\sqrt{f}(y)\delta^2(h,\lambda x,\lambda y)}{|(x-y)^t A_i(\lambda x,\lambda y)(x-y)|^{N+2s}}dydx
\end{align*}
Since $h$ is smooth we have that 
\[
\delta^2(h,\lambda x,\lambda y)\leq C\lambda^2|x-y|^2,
\] 
and since $\sqrt f$ has compact support, we conclude
\[
\int_{\R^{2N}}\sqrt{f_\lambda}(x)\sqrt{f_\lambda}(y)\delta^2(h,x,y)K_i(x,y)dydx\leq C\lambda^{2-2s}
\]

Finally we need to analyze the error term, note first that 
\[
\inf\limits_{w\in C^{\R^N}}c_{N,s}\int_{\R^{2N}}\sqrt{f}(x)\sqrt{f}(y)Q(h,w,x,y)K(x,y)dxdy\geq 0,
\]  
Now let $w$ be any smooth fixed function and notice that 
\begin{align*}
&\hspace{-40pt}\int_{\R^{2N}}\sqrt{f_\lambda}(x)\sqrt{f_\lambda}(y)Q(h,w,x,y)K(x,y)dxdy\\
&\hspace{80pt}=\lambda^{-N}\int_{\R^{2N}}\frac{\sqrt{f}(x/\lambda)\sqrt{f}(y\lambda)Q(h,w,x,y)}{|(x-y)^t A_i( x, y)(x-y)|^{N+2s}}dydx\\
&\hspace{80pt}=\lambda^{N}\int_{\R^{2N}}\frac{\sqrt{f}(x)\sqrt{f}(y)Q(h,w,\lambda x,\lambda y)}{|(\lambda x-\lambda y)^t A_i(\lambda x,\lambda y)(\lambda x-\lambda y)|^{N+2s}}dydx\\
&\hspace{80pt}=\lambda^{-2s}\int_{\R^{2N}}\frac{\sqrt{f}(x)\sqrt{f}(y)Q(h,w,\lambda x,\lambda y)}{|(x-y)^t A_i(\lambda x,\lambda y)(x-y)|^{N+2s}}dxdy.
\end{align*}
As before, observe that $Q(h,w,\lambda x,\lambda y)\leq C\lambda^2|x-y|^2$ and since $\sqrt f$ has compact support, we deduce that $\mathcal E= o(\lambda^{2-2s})$.

Taking all the above inequalities into account we deduce that
\[
\lim\limits_{\lambda\to 0}\lambda^{2s}I_i=\int_{\R^{2N}}\frac{\sqrt f (x)-\sqrt f (y))^2}{|(x-y)^t A_i(0,0)(x-y)|^{N+2s}}dydx.
\]
To deduce the same inequality at any $x_0$ we just consider 
\[
f_\lambda(x)=\lambda^{-N}f\left(\frac{x-x_0}{\lambda}\right)
\] and apply the previous argument.
\end{proof}

The previous lemma states that, under the hypothesis of Theorem \ref{mainteo}, the energy functional associated to the pure diffusion must coincide on the set of functions $\{g\in C^2(\R^N),\ g=\sqrt f\}$, where $f$ is a smooth probability density with compact support. 

The next lemma recovers pointwise information on the kernels, based on the previous energy equality. The proof uses the fact that the energy can be rewritten using the Fourier transform.



%

\begin{lemma}\label{Fourier}
Let $A_i\in \mathbb S^{N}$, with $i=1,2$, be constant positive definite matrices and denote 
\[
K_i(x,y) = \frac{1}{|(x-y)^t A_i (x-y)|^{(N + 2s)/2}}.
\] 
If
$$
\int \cB_{K_1}(\sqrt{f})dx = \int \cB_{K_2}(\sqrt{f})dx
$$
for all nonnegative functions $f \in C_0^\infty$ such that $\sqrt{f} \in C_0^{2s\alpha}$ for some $\alpha>0$, then $A_1 = A_2$. 
\end{lemma}

\begin{proof}
Denote $g = \sqrt f$, where $f$ is a smooth probability density with compact support. By hypothesis we also have that $g$ is smooth. For the rest of the proof let $A$ be symmetric positive definite matrix, write $A = PDP^{t}$ and denote $B = P \sqrt{D}$. Observe also that $A = BB^t$. Denote by $K(x,y)=|(x-y)^t A (x-y)|^{-(N + 2s)/2}$ We have
\begin{align*}
\int \cB_K(g,g) dx = \mathrm{Det}(A)^{-1} [g_{B}]_{H^s}^2,
\end{align*} 
where $g_{B}(x) = g(B^{-t} x)$, where $B^{-t}  = (B^{-1})^t$. Using Proposition 4.2 in~\cite{Hitchhikers}, we get that
\begin{align*}
[g_{B}]_{H^s}^2 = \int |\xi|^{2s} |\hat g_{B}(\xi)|^2 d\xi,
\end{align*}
where $\hat g$ denotes the Fourier transform of $g$. Therefore 
\begin{align*}
\int \cB_K(g,g) dx & = \mathrm{Det}(A)^{-1} \int |\xi|^{2s} \Big{|} \int e^{i \langle x, \xi\rangle} g_{B}(x)dx \Big{|}^2 d\xi \\
& = \mathrm{Det}(A)^{-1}\int |\xi|^{2s} |\int e^{i \langle x, \xi\rangle} g(B^{-t} x)dx|^2 d\xi \\
& = \mathrm{Det}(A)^{-1} |\mathrm{Det}(B)|^2 \int |\xi|^{2s} |\int e^{i \langle B^t y, \xi\rangle} g(y)dy|^2 d\xi \\
& =  \int |\xi|^{2s} |\int e^{i \langle   y,  B \xi\rangle} g(y)dy|^2 d\xi.
\end{align*} 

Make the change $\tilde \xi = B \xi$ to conclude
\begin{align*}
\int \cB_K(g,g) dx  & = |\mathrm{Det}(B^{-1})| \int |B^{-1}\xi|^{2s} |\int e^{i \langle y, \xi\rangle} g(y)dy|^2 d\xi \\
& = |\mathrm{Det}(A)|^{-1/2} \int |B^{-1}\xi|^{2s} \hat g^2(\xi) d\xi.
\end{align*}

Since $|B^{-1} \xi|^2 = \langle B^{-1} \xi, B^{-1} \xi \rangle = \langle A^{-1} \xi, \xi \rangle$ we conclude that
\begin{align}\label{faith}
\int \cB_A(g,g) dx  & = |\mathrm{Det}(A)|^{-1/2} \int \langle A^{-1} \xi, \xi \rangle^{s} \hat g^2(\xi) d\xi.
\end{align}

Now, consider a function $g$ with $g(x) = g_1(x_1/\lambda) g_2(x')$ where $x = (x_1, x')$. Then, writing $A^{-1} = (A^{ij})_{ij}$ and 
\begin{align*}
A^{-1} = \left [  \begin{array}{cc} A^{11} & v^t \\ v & (A^{-1})' \end{array}\right ], 
\end{align*}
we get
\begin{align*}
& \int \langle A^{-1} \xi, \xi \rangle^{s} \hat g^2(\xi) d\xi \\
= & \lambda^2 \int \langle A^{-1} \xi, \xi \rangle^{s} \hat g_1^2(\lambda \xi_1) \hat g^2(\xi') d\xi \\
= & \lambda^2 \int \Big{(} A^{11} \xi_1^2 + 2 \xi_1 \langle v, \xi'  \rangle + \xi' (A^{-1})' \xi' \Big{)}^s \hat g_1^2(\lambda \xi_1) \hat g_2^2(\xi') d\xi \\
= & \lambda^{1-2s} \int \Big{(} A^{11} \xi_1^2 + 2 \lambda \xi_1 \langle v, \xi'  \rangle + \lambda^2 \xi' (A^{-1})' \xi' \Big{)}^s \hat g_1^2(\xi_1) \hat g_2^2(\xi') d\xi
\end{align*}

Then, after a normalization in $\lambda$ and taking $\lambda \to 0$, we conclude that
\begin{align*}
\lim_{\lambda \to 0} \lambda^{2s - 1} \int \cB_A(g,g) dx  & = 
|\mathrm{Det}(A)|^{-1/2} \int \Big{(} A^{11} \xi_1^2 \Big{)}^s \hat g_1^2(\xi_1) \hat g_2^2(\xi') d\xi \\
& = C |\mathrm{Det}(A)|^{-1/2} (A^{11})^s
,
\end{align*}
with 
$$
C = [g_1]_{H^s(\R)} |g_2|_{L^2(\R^{N - 1})}.
$$ 

Now let $A_i$ and $K_i$ as in the statement of the lemma. By hypothesis we have
\begin{equation*}
\int \cB_{K_1}(g,g) dx = \int \cB_{K_2}(g,g) dx,
\end{equation*}
and therefore we conclude
\begin{align*}
|\mathrm{Det}(A_1)|^{-1/2} (A_1^{11})^s = |\mathrm{Det}( A_2)|^{-1/2} (A_2^{11})^s.
\end{align*}

Notice that if $E$ is a change of rows matrix and if $A$ is a positive definite matrix, then $\mathrm{Det}(E A E) = \mathrm{Det}(A)$ since $\mathrm{Det}(E) = -1$. For $k = 2,...,N$, we consider $E_k$ the matrix exchanging the first and k-th rows of the identity matrix. Let $A_k = E_k A E_k$ we have $A_k^{-1} = E_k A^{-1} E_k$ and therefore $(A_k^{-1})_{11} = A^{kk}$. We may apply the procedure done after \eqref{faith} to $A_i$, (replacing $g(x)$ by $g(E_k x)$) to conclude that
\begin{align*}
|\mathrm{Det}(A_1)|^{-1/2} (A_1^{kk})^s = |\mathrm{Det}(A_2)|^{-1/2} ({A_2}^{kk})^s, \quad \mbox{for all} \ k = 1,...,N.
\end{align*}

Denote 
\begin{align}\label{faith2}
\rho =  \left( \frac{|\mathrm{Det}(A_1)|}{|\mathrm{Det}(A_2)|} \right)^{\frac{1}{2s}},
\end{align}
to conclude that
\begin{align}\label{faith1}
\frac{A_1^{ii}}{A_2^{ii}} = \rho, \quad \mbox{for all} \ i.
\end{align}

\medskip

Now, for $k, m \in \{ 1,...,N \}$ with $k < m$, let $E = E(k,m)$ be the rotation matrix in $\pi/4$ over the plane $(k,m)$, that is
\begin{equation*}
E_{ij} = \left \{ \begin{array}{cl} \cos(\pi/4) \quad & \mbox{if} \ i=j=k, \ \mbox{or} \ i=j=m \\ 
-\sin(\pi/4) \quad & \mbox{if} \ i=k, \ j=m \\
\sin(\pi/4) \quad & \mbox{if} \ i=m, \ j = k \\
\delta_{ij} \quad & \mbox{in other case} \end{array} \right .
\end{equation*}

Let $\hat A = E A E^t$ and recall that $E E^t = I$. We also have that $\mathrm{Det}(\hat A) = \mathrm{Det}(A)$ and in addition we have the identity $\hat A^{-1} = E A^{-1} E^t$. Therefore
\begin{align*}
\hat A^{kk} 
& = \sum_{\alpha, \beta} A^{\alpha \beta} E_{k \beta} E_{k\alpha} \\
& = \sum_{\alpha} E_{k\alpha}\Big{(} A^{\alpha k } E_{kk} + A^{\alpha m } E_{km}\Big{)} \\
& = E_{kk}\Big{(} A^{k k } E_{kk} + A^{k m } E_{km}\Big{)}
+ E_{km}\Big{(} A^{m k } E_{kk} + A^{m m } E_{km}\Big{)} \\
& = \frac{1}{2} \Big{(} A^{kk} - 2 A^{km} + A^{mm} \Big{)}.
\end{align*}

We may apply this fact combined with  a similar argument as before (replacing $g$  by $g_E(x) := g(Ex)$) to $A_i$ to conclude that
\begin{equation*}
\frac{A_1^{kk} - 2 A_1^{km} + A_1^{mm}}{A_2^{kk} - 2 A_2^{km} + A_2^{mm}}= \frac{\hat A_1^{kk}}{\hat A_2^{kk}} = \rho, \quad \mbox{for all} \ k.
\end{equation*}
Since by \eqref{faith1} we have $A_1^{ii} = \rho A_2^{ii}$ for all $i$, we conclude that
$$
A_1^{ij} = \rho A_2^{ij}, \quad \mbox{for all} \ i, j.
$$

Thus, we conclude that $A_1 = \rho A_2$. This equality in conjunction with \eqref{faith2} let us arrive at
\begin{equation*}
\rho = (\rho^{N})^{\frac{1}{2s}},
\end{equation*}
from which $\rho = 1$ and therefore $A_1 = A_2$.
	\end{proof}

At this point we are able to deduce the first conclusion of Theorem \ref{mainteo}. We state it as the following lemma.
\begin{corollary}\label{extra}
Under the hypothesis of Theorem \ref{mainteo} we have
\[
A_1(x_0,x_0)=A_2(x_0,x_0).
\]
If furthermore the following decomposition holds
\begin{align}\label{sum}
A_i(x,y)=\tilde A_i(x)+\tilde A_i(y)
\end{align}
or\eqref{sproduct} then
\[
\tilde A_1(x)= \tilde A_2(x), \quad \forall x\in\R^N
\]
In particular 
\[
\int_{\R^N}\cB_{K_1}(\sqrt f)(x)dx=\int_{\R^N}\cB_{K_2}(\sqrt f)(x)dx
\]
\end{corollary}
\begin{proof}
Thanks to Lemma \ref{difusionigual} and Lemma \ref{Fourier} it is easy to conclude $A_1(x_0,x_0)=A_2(x_0,x_0)$. From here and the extra hypothesis \eqref{product} or \eqref{sum} the conclusion of the corollary follows by noting that since $\tilde A$ is a symmetric positive definite matrix its square root is well defined.
\end{proof}
At this point we have concluded the first part of Theorem \ref{mainteo}. The second part of Theorem \ref{mainteo} is proven in the next lemma.
\begin{lemma}
Under the hypothesis of Theorem \ref{mainteo} and Corollary \ref{extra} we have that
\[
L_Kh_1=L_Kh_2\quad \text{in }\R^N, 
\]
where $K$ is the kernel in Lemma \ref{difusionigual}. As a consequence, we deduce $h_1=h_2+C$ for some constant $C$.
\end{lemma}

\begin{proof}
We follow the steps of the proof of Lemma \ref{difusionigual}. Observe that thanks to Corollary \ref{extra} we have that 
\[
\int_{\R^N}\cB_{K_1}(\sqrt f)(x)dx=\int_{\R^N}\cB_{K_2}(\sqrt f)(x)dx
\]
and therefore, applying the same arguments as in the proof \ref{difusionigual} we have that
\[
\lim\limits_{\lambda\to 0}I_i=\int_{\R^N}f(x)L_{K_i}h_i(0)dx.
\]
With this and given that $I_1=I_2$ we conclude
\[
\int_{\R^N}f(x)L_{K_i}h_i(0)dx=\int_{\R^N}f(x)L_{K_i}h_i(0)dx.
\]
We can reproduce this argument at every $x\in\R^N$ to deduce 
\[
\int_{\R^N}f(x)L_{K_i}h_i(x)dx=\int_{\R^N}f(x)L_{K_i}h_i(x)dx.
\]
From here, we conclude 
\[
L_Kh_1(x)=L_Kh_2(x), \quad x\in\R^N.
\]
Since $h$ is bounded we can apply now Liouville's theorem, see Lemma \ref{Liouville} to deduce that there exists a constant $C$ such that $h_1=h_2+C$.

\end{proof}

\section{Appendix}

\begin{lemma}\label{lemarescale}
Assume that $\partial\Omega$ is of class $C^2$. There exists a constant $c_{\pm}, \delta > 0$ just depending on $N, s, \gamma, \Gamma$ and $\alpha$ such that
	\begin{equation*}
	c_- d^{\alpha - 2s} \leq L_K d^{\alpha} + \cB(h, d^\alpha) \leq c_+ d^{\alpha - 2s} \quad \mbox{in} \ \Omega_{\delta}.
	\end{equation*}
	
	If $\alpha > s$, then $c_- > 0$; and if $\alpha < s$, then $c_+ < 0$.
\end{lemma}

\begin{proof}
	We argue by contradiction as in~\cite{BDPGMQ}. We assume the existence of a sequence $x_k \to \partial \Omega$ such that
	$$
	\limsup_{k \to \infty} d_k^{-\gamma + 2s} \Big{(}L_Kd^{\gamma}(x_k) + \cB(h, d^\gamma)(x_k) \Big{)}< c_0,
	$$
	where $d_k = d(x_k)$. By rotation and translation, we can assume $x_k \to x_0 \in \partial \Omega$ such that the unit inward normal $n(x_0) := Dd(x_0)$ equals $e_1$, the first canonical vector in $\R^N$. In that case, by the smoothness of the domain, Dominated Convergence Theorem leads to
	\begin{equation*}
	d_k^{-\alpha + 2s} L_K d^\alpha(x_k) \to \pv \int_{\R^N} [(1 + y_1)_+^\alpha - 1] \frac{dy}{|y^t A(x_0, x_0) y|^{(N + 2s)/2}}, \quad \mbox{as} \ k \to \infty,
	\end{equation*}
	where $y = (y_1, y') \in \R \times \R^{N - 1}$. From here we denote $A= A(x_0, x_0)$ for simplicity.
	
	We concentrate in the integral
	\begin{equation}\label{defJ}
	J = \int_{\R^{N - 1}} \frac{dy'}{|y^t Ay|^{(N + 2s)/2}},
	\end{equation}
	starting with the decomposition
	$$
	A = \left [  \begin{array}{cc} a_{11} & v^t \\ v & A' \end{array}\right ],
	$$
	where $a_{11} > 0$, $v \in \R^{N - 1}$ and $A'$ is a $(N - 1) \times (N - 1)$ symmetric, positive definite matrix. Denoting 
	$$
	A_1 = a_{11}^{-1}A, \quad v_1 = a_{11}^{-1} v \quad \mbox{and} \quad A_1' = a_{11}^{-1} A',
	$$
	we have 
	\begin{align*}
	y^t A y = & a_{11}y_1^2 + 2 (y_1, 0')^t A (0, y') + (0, y')^t A (0, y') \\
	= & a_{11} y_1^2 \Big{(} 1 + 2 (1, 0')^t A_1 (0, \frac{y'}{y_1}) + (0, \frac{y'}{y_1})^t A_1 (0, \frac{y'}{y_1})\Big{)}
	\end{align*}

	Thus, performing the change of variables $z' = y'/y_1$ we have
	\begin{align*}
	J = & a_{11}^{-(N + 2s)/2}|y_1|^{-(1 + 2s)} \int_{\R^{N - 1}} \frac{dz'}{|1 + 2 (1, 0')^t A_1 (0,z') + (0, z')^t A_1 (0, z')|^{(N + 2s)/2}} \\
	= & a_{11}^{-(N + 2s)/2}|y_1|^{-(1 + 2s)} \int_{\R^{N - 1}} \frac{dz'}{|1 + 2 v_1^t z' + z'^t A_1' z'|^{(N + 2s)/2}}.
	\end{align*}

	Now, we use the following decomposition of the matrix $A_1$:
	$$
	A_1 = \left [  \begin{array}{cc} 1 & v_1^t \\ v_1 & B B^t \end{array}\right ],
	$$
	that is, $A_1' = BB^t$ where $B$ is an $(N-1) \times (N-1)$ invertible matrix. Then, using the change of variables $t' = B^t z'$, we have that
	\begin{align}\label{sabado}
	J = a_{11}^{-(N + 2s)/2}|y_1|^{-(1 + 2s)} \int_{\R^{N - 1}} \frac{ |\mathrm{Det}(B)|^{-1}dt'}{|1 + 2 (B^{-1}v_1)^t t' + |t'|^2|^{(N + 2s)/2}}.
	\end{align}
	
	Thus, if $B^{-1}v_1 = 0$ (that is, $v_1 = 0$ and $\mathrm{Det}(A_1) = \mathrm{Det}(A_1')$), we conclude that
	\begin{align*}
	J = & a_{11}^{-(N + 2s)/2}|y_1|^{-(1 + 2s)} |\mathrm{Det}(A_1)|^{-1/2} C^*,
	\end{align*}
	with
	$$
	C^* = \int_{\R^{N - 1}} \frac{ dt'}{(1 + |t'|^2)^{(N + 2s)/2}}.
	$$
	
	Using the definition of $A_1$ we conclude that
	\begin{align}\label{pandemia2}
	J = & a_{11}^{-s}|y_1|^{-(1 + 2s)} |\mathrm{Det}(A)|^{-1/2} C^*,
	\end{align}

	If $|B^t v_1| =: c_1 \neq 0$ we use a rotation driving $B^{-1}v_1$ to $|B^{-1} v_1| e'_1$ where $e'_1$ is the first canonical vector in $\R^{N-1}$. Thus, by the symmetry of the domain of integration we conclude from~\eqref{sabado} that
	\begin{align*}
	J = a_{11}^{-(N + 2s)/2}|y_1|^{-(1 + 2s)} \int_{\R^{N - 1}} \frac{ |\mathrm{Det}(B)|^{-1}dt'}{|1 + 2 c_1 t_1' + |t'|^2|^{(N + 2s)/2}},
	\end{align*}
	where $t' = (t_1', t'')$.
	
	At this point, we use the property that
	$$
	\mathrm{Det}(A_1) = \mathrm{Det}(A_1') (1 - v_1^t (BB^t)^{-1} v_1) = \mathrm{Det}(A_1') (1 - |B^{-1}v_1|^2) =  \mathrm{Det}(A_1') (1 - c_1^2),
	$$
	from which we see that $c_1 \in (0,1)$. We write
	\begin{align*}
	1 + 2 c_1 t_1' + |t'|^2 = & (1 - c_1^2) + (t_1' + c_1)^2 + |t''|^2 \\
	= & (1 - c_1^2)\Big{(} 1 + \Big{(} \frac{t_1'}{1 - c_1^2} + \frac{c_1}{1 - c_1^2} \Big{)}^2 + \Big{|} \frac{t''}{1 - c_1^2} \Big{|}^2 \Big{)},
	\end{align*}
	from which, using homogeneity and translation invariance properties on the integral, we arrive at
	\begin{align*}
	\int_{\R^{N - 1}} \frac{dt'}{|1 + 2 c_1 t_1' + |t'|^2|^{(N + 2s)/2}} = (1 - c_1^2)^{-(1 + 2s)/2} C^*,
	\end{align*}
	with $C^*$ as above. Thus, replacing in $J$ we get
	\begin{align*}
	J = & a_{11}^{-(N + 2s)/2}|x_1 - y_1|^{-(1 + 2s)} |\mathrm{Det}(A_1')|^{-1/2} (1 - c_1^2)^{-(1 + 2s)/2} C^* \\
	= & a_{11}^{-(1 + 2s)/2}|x_1 - y_1|^{-(1 + 2s)} |\mathrm{Det}(A')|^{-1/2} \Big{(} \frac{\mathrm{Det}(A_1)}{\mathrm{Det}(A_1')}\Big{)}^{-(1 + 2s)/2} C^*,
	\end{align*}
	and from this we conclude that
	\begin{align}\label{pandemia3}
	J = 
	& |y_1|^{-(1 + 2s)} |\mathrm{Det}(A')|^{s} |\mathrm{Det}(A)|^{-(1 + 2s)/2} C^*,
	\end{align}
	which coincides with~\eqref{pandemia2} in the previous case. Then, we replace in~\eqref{defJ} from which we get that
	\begin{equation*}
	d_k^{-\alpha + 2s} L_K d^\alpha(x_k) \to C \ \pv \int_{\R} [(1 + y_1)_+^\alpha - 1] \frac{dy_1}{|y_1|^{1 + 2s}}, \quad \mbox{as} \ k \to \infty,
	\end{equation*}
	for some explicit constant $C > 0$ just depending on $N, s, \gamma, \Gamma$. Thus, 
	\begin{equation*}
	d_k^{-\alpha + 2s} L_K d^\alpha(x_k) \to C_0, \quad \mbox{as} \ k \to \infty,
	\end{equation*}
	for some $C_0 \in \R$. 
	
	On the other hand, a simple computation using the smoothness of $h$ shows that
	$$
	\cB_K (h, d^\alpha)(x_k) = O(d_k^{\alpha - 2s + 1}),
	$$
	from which we conclude that
	\begin{equation*}
	d_k^{-\alpha + 2s} (L_K d^\alpha(x_k) + \cB_K(h, d^\alpha)(x_k)) \to C_0, \quad \mbox{as} \ k \to \infty,
	\end{equation*}
	but this is a contradiction in view of the results in...
\end{proof}

\begin{lemma}\label{Liouville}
Let $u$ be a bounded classical solution of 
\begin{align}\label{eqL}
L_K=0, \quad\text{in }\R^N,
\end{align}
where 
\[
K(x,y)=\frac{1}{|(x-y)^tA(x,y)(x-y)|^{(N+2s)/2}}
\]
and $A(x,y)$ is elliptic, that is 
\[
\gamma|\xi|^2<\xi^t A(x,y)\xi<\Gamma|\xi|^2,
\]
for some $\Gamma>\gamma>0$. Then $u$ is constant.
\end{lemma}
\begin{proof}
Without loss of generality we can assume
\[
\inf\limits_{\R^N}u=0.
\]
If $u\equiv0$ there is nothing to prove. Otherwise let $R>1$, define $u_R=u(Rx)$, $A_R(x,y)=A(Rx, Ry)$ and its associated kernel 
\[
K_R(x,y)=\frac{1}{|(x-y)^tA_R(x,y)(x-y)|^{(N+2s)/2}}.
\]
Observe now that $A_R$ is elliptic with the same ellipticity constants of the original matrix $A$. Furthermore, since $u$ is a solution of \eqref{eqL}
\begin{align*}
L_{K_R}u_R(x)&=\text{P.V.}\int\frac{u(Rx)-u(Ry)}{|(x-y)^tA_R(x,y)(x-y)|^{(N+2s)/2}}dy\\
&=R^{2s}\text{P.V.}\int\frac{u(Rx)-u(z)}{|(Rx-z)^tA_R(x,z/R)(Rx-z)|^{(N+2s)/2}}dy\\
&=R^{2s}\text{P.V.}\int\frac{u(Rx)-u(z)}{|(Rx-z)^tA(Rx,z)(Rx-z)|^{(N+2s)/2}}dy\\
&=R^{2s}L_Ku(Rx)\\
&=0.
\end{align*}
Let $\varepsilon>0$ and note that 
\[
\inf\limits_{B_1^c} u_R\leq \varepsilon,
\]
otherwise, by the strong maximum principle, $u_R\equiv0$. By the minimum principle then
\[
\inf\limits_{B_1} u_R\leq \varepsilon,
\] 
Then by the Harnack inequality, see for example \cite{Caffarelli-Silvestre}, there exists a universal constant $C=C(n,\gamma, \Gamma)$, but not depending on $R$, such that 
\[
\sup\limits_{B_{1/2}}u_R\leq C\inf\limits_{B_1} u_R\leq C\varepsilon,
\]
which translates into
\[
\sup\limits_{B_{R/2}}\leq C\varepsilon.
\]
Since $\varepsilon$ and $R$ are arbitrary we deduce that $u=0$.
\end{proof}

\bigskip

\noindent {\bf Acknowledgements.} 

G. D. was partially supported by Fondecyt Grant No. 1190209.

E. T. was partially supported by Fondecyt No. 1201897.


\begin{thebibliography}{00}

\bibitem{BE}
Bakry, D.; \'Emery, M. {\em Diffusions hypercontractives.} In S\'eminaire de Probabilit \'es, XIX, 1983/84.Lecture Notes in Math. 1123 177-206. Springer, Berlin.

\bibitem{BDS}
Banerjee, A., D\'avila, G., Sire, Y. {\em Regularity for parabolic systems with critical growth in the gradient and applications.} arXiv:2005.04004





\bibitem{BDPGMQ}
Barrios, B., Del Pezzo, L., Garc\'ia-Meli\'an, J. and Quaas, A. {\em A Priori Bounds and Existence of Solutions for Some Nonlocal Elliptic Problems.} Rev. Matem\'atica Iberoamericana, Volume 34, Issue 1, 2018, pp. 195-220.

\bibitem{BNV}
Berestycki, H., Nirenberg, L. and Varadhan, S. {\em The principal eigenvalue and maximum principle for second order elliptic operators in general domains.} Comm. Pure Appl. Math. 47 (1) (1994) 47-92.


\bibitem{BD}
Birindelli, I., and Demengel, F. {\em First eigenvalue and maximum principle for fully nonlinear singular operators.} Adv. Differential Equations, Vol. 11, no. 1, (2006) 91-119.


\bibitem{Caffarelli-Silvestre}
Caffarelli, Luis; Silvestre, Luis. {\em Regularity theory for fully nonlinear integro-differential equations.} Comm. Pure Appl. Math. 62 (2009), no. 5, 597-638.

\bibitem{C}
Cozzi, M. {\em Interior regularity of solutions of non-local equations in Sobolev and Nikol'skii spaces.} Annali di Matematica (2017) 196: 555-578.


%



%
%
%
%
%
\bibitem{DQT}
D\'avila, G., Quaas, A. and Topp, E. {\em Existence, Nonexistence and Multiplicity Results for Fully Nonlinear Nonlocal Dirichlet Problems.} J. Differential Equations 266 (2019), no. 9, 5971-5997. 

\bibitem{DR}
Da Lio, F.; Rivi\`ere, T. {\it Three-term commutator estimates and the regularity of 1/2-harmonic maps into spheres.} Anal. PDE 4 (2011), no. 1, 149-190. 

\bibitem{DR2}
Da Lio, F.; Rivi\`ere, T. {\it Sub-criticality of non-local Schr\"odinger systems with antisymmetric potentials and applications to half-harmonic maps.} Adv. Math. 227 (2011), no. 3, 1300-1348.

\bibitem{DV}
Donsker, M.D.; Varadhan S.R.S {\em On the principal eigenvalue of second-order elliptic differential operators.} Comm. Pure Appl. Math., 29: 595-621. doi:10.1002/cpa.3160290606

\bibitem{D}
Dyda, B. {\em Fractional calculus for power functions and eigenvalues of the fractional Laplacian.}  Fract. Calc. Appl. Anal. 15 (2012), no. 4, 536-555.

\bibitem{GT}
Gilbarg, D. ; Trudinger, N.S. {\em Elliptic partial differential equations of second order.} Springer-Verlag, Berlin 2001.

\bibitem{Hitchhikers}
Di Nezza, E.; Palatucci, G.; Valdinoci, E. {\em Hitchhiker's guide to the fractional Sobolev spaces.} Bull. Sci. Math. 136 (2012), no. 5, 521-573.

\bibitem{LPPS}
Leonori, T.; Peral, I.; Primo, A.; Soria, F. {\em Basic estimates for solutions of a class of nonlocal elliptic and parabolic equations.} Disc. Cont. Dyn. Syst. Volume 35, no. 12 (2015) 6031-6068.

%
%
%



\bibitem{MR}
Ma, Z. M.; R\"ockner, M. {\em Introduction to the theory of (nonsymmetric) Dirichlet forms.} Universitext. Springer-Verlag, Berlin, 1992. vi+209 pp. 

\bibitem{Millot-Sire}
Millot, Vincent; Sire, Yannick. {\em On a fractional Ginzburg-Landau equation and 1/2-harmonic maps into spheres.} Arch. Ration. Mech. Anal. 215 (2015), no. 1, 125-210.
%
%
%
%

\bibitem{RS}
Runst, T.; Sickel, W. {\em Sobolev Spaces of Fractional Order, Nemytskij Operators, and Nonlinear Partial Differential Equations.} De Gruyter Series on Nonlinear Analysis and Applications 3, 1996.

\bibitem{S}
Sion, M. {\em On general minimax theorems.} Pac. J. Math., vol. 8 (1958) 171-176.




\bibitem{SWZ}
Spener, A.; W., F.; Zacher, R. {\em The fractional Laplacian has infinite dimension.} Comm. Partial Differential Equations 45 (2020), no. 1, 57-75. 
\end{thebibliography}
\end{document}